\documentclass[table]{amsart}
\usepackage[margin=1.1in]{geometry}
\usepackage{amscd,amsmath,amsxtra,amsthm,amssymb,stmaryrd,xr,mathrsfs,mathtools,enumerate,commath, comment, mathtools, orcidlink}
\usepackage{tikz}
\usetikzlibrary{arrows.meta,positioning}

\usetikzlibrary{calc}
\usepackage{tikz-3dplot}
\usepackage{stmaryrd}
\usepackage{multirow}
\usepackage{xcolor}
\usepackage{commath}
\usepackage{comment}
\usepackage{svg}
\usepackage{graphics}
\usepackage{longtable} 
\usepackage{pdflscape} 
\usepackage{booktabs}
\usepackage{hyperref}
\definecolor{vegasgold}{rgb}{0.77, 0.7, 0.35}
\definecolor{darkgoldenrod}{rgb}{0.72, 0.53, 0.04}
\definecolor{gold(metallic)}{rgb}{0.83, 0.69, 0.22}
\hypersetup{
 colorlinks=true,
 linkcolor=darkgoldenrod,
 filecolor=brown,      
 urlcolor=gold(metallic),
 citecolor=darkgoldenrod,
 }
\newtheorem{lthm}{Theorem}

\usepackage[all,cmtip]{xy}

\DeclareFontFamily{U}{wncy}{}
\DeclareFontShape{U}{wncy}{m}{n}{<->wncyr10}{}
\DeclareSymbolFont{mcy}{U}{wncy}{m}{n}
\DeclareMathSymbol{\Sh}{\mathord}{mcy}{"58}
\usepackage[T2A,T1]{fontenc}
\usepackage[OT2,T1]{fontenc}

\newtheorem{theorem}{Theorem}[section]
\newtheorem{lemma}[theorem]{Lemma}

\newtheorem*{theorem*}{Theorem}
\newtheorem*{ass*}{Assumption}
\newtheorem{definition}[theorem]{Definition}

\newtheorem{remark}[theorem]{Remark}

\newtheorem{proposition}[theorem]{Proposition}

\newcommand{\GL}{\operatorname{GL}}

\newcommand{\Ind}{\mathrm{Ind}}
\newcommand{\diag}{\mathrm{diag}}

\newcommand{\St}{\mathrm{St}}
\newcommand{\Z}{\mathbb{Z}}
\newcommand{\Q}{\mathbb{Q}}
\newcommand{\F}{\mathbb{F}}

\newcommand{\op}[1]{\operatorname{#1}}

 \DeclareMathSymbol{\sha}{\mathord}{mcy}{"58}
 \makeatletter
\newcommand{\mylabel}[2]{#2\def\@currentlabel{#2}\label{#1}}
\makeatother

\numberwithin{equation}{section}

\title[Asymptotic Vanishing of Stiefel--Whitney Classes for $\GL_n(\F_q)$]{Asymptotic Vanishing of Stiefel--Whitney Classes for $\GL_n(\F_q)$}
\author[A.~Ray]{Anwesh Ray\, \orcidlink{0000-0001-6946-1559}}
\address[Ray]{Chennai Mathematical Institute, H1, SIPCOT IT Park, Kelambakkam, Siruseri, Tamil Nadu 603103, India}
\email{anwesh@cmi.ac.in}

\begin{document}
\maketitle

\begin{abstract}
We study the asymptotic behavior of Stiefel--Whitney classes of irreducible orthogonal representations of the finite general linear groups $\GL_n(\F_q)$. Building on recent formulas expressing these classes in terms of character values at elements of order dividing $2$, we relate questions about characteristic classes to problems of $2$-adic divisibility of character values. For fixed odd $q$, we show that as $n \to \infty$, the values of irreducible orthogonal characters become highly divisible by powers of $2$ for almost all representations. As a consequence, the proportion of irreducible orthogonal representations with trivial first and second Stiefel--Whitney classes tends to $1$, and if $q \equiv 1 \pmod{4}$, the same holds for the fourth Stiefel--Whitney class. In particular, almost all orthogonal representations are spinorial in the large rank limit. In contrast, when the rank is fixed and $q \to \infty$, the behavior is markedly different. Focusing on $\GL_2(\F_q)$, we show that the second Stiefel--Whitney class vanishes with limiting probability $3/8$ among irreducible orthogonal representations.
\end{abstract}
\section{Introduction}

The study of characteristic classes of group representations occupies a central position at the interface of algebra, topology, and representation theory. Among these, the Stiefel--Whitney classes of real representations encode subtle arithmetic and geometric information, ranging from lifting problems to the topology of classifying spaces. In recent years, there has been growing interest in understanding these invariants for representations of finite groups of Lie type, particularly through explicit formulas in terms of character values on elements of small order (see for instance \cite{JoshiSpallone1,JoshiSpallone2,GangulyJoshi2,GangulyJoshiTotalSWCs, MalikSpallone2}).

A general philosophy emerging from topology suggests that low-degree characteristic classes of large-dimensional objects tend to exhibit vanishing phenomena. This principle has been made precise in several settings. For instance, in the case of symmetric groups, results of Ayyer--Prasad--Spallone \cite{APS} and Ganguly--Spallone \cite{GangulySpallone} show that, asymptotically, almost all irreducible representations are not only achiral (i.e.\ have trivial first Stiefel--Whitney class), but are in fact spinorial, with both $w_1$ and $w_2$ vanishing with probability tending to $1$ as the rank grows.

The purpose of this paper is to establish analogous asymptotic vanishing results for irreducible orthogonal representations of the finite general linear groups $\GL_n(\F_q)$, as the rank $n \to \infty$ with $q$ fixed. Our approach builds on recent work expressing Stiefel--Whitney classes of real representations in terms of character values evaluated at elements of order dividing $2$. These formulas reduce questions about characteristic classes to problems concerning congruences and divisibility properties of character values. The main technical input of the paper is a family of asymptotic divisibility results for irreducible characters of $\GL_n(\F_q)$. More precisely, we show that for any fixed element $g$ in a smaller general linear group, the values of irreducible orthogonal characters evaluated at $g$ become highly divisible by powers of $2$ for almost all representations as $n \to \infty$. This leverages techniques developed by Shah and Spallone \cite{shah2024divisibility} who study asymptotic divisibility questions for all irreducible representations of $\op{GL}_n(\F_q)$ (as $n\rightarrow \infty$) and ultimately relies on the combinatorial structure arising from Green's parametrization of irreducible representations \cite{Green}.

Combining these asymptotic divisibility results with the explicit formulas for Stiefel--Whitney classes, we obtain our first main theorem, stated below.

\begin{lthm}[Theorem \ref{thm asymp SW}]
Let $q$ be an odd prime power and, for each $n \ge 1$, let $G_n = \GL_n(\F_q)$. Denote by $\op{OIrr}(G_n)$ the set of irreducible complex representations of $G_n$ which admit a $G_n$-invariant real structure (equivalently, irreducible orthogonal representations). For $\pi \in \op{OIrr}(G_n)$, let $w_i(\pi) \in H^i(G_n;\Z/2\Z)$
denote its $i$th Stiefel--Whitney class.
\begin{enumerate}
\item The proportion of irreducible orthogonal representations with trivial first and second Stiefel--Whitney classes tends to $1$ as $n \to \infty$, i.e.,
\[
\lim_{n\to\infty}
\frac{
\#\{\pi \in \op{OIrr}(G_n) : w_1(\pi)=w_2(\pi)=0\}
}{
|\op{OIrr}(G_n)|
}
= 1.
\]

\item If $q \equiv 1 \pmod{4}$, then the proportion of irreducible orthogonal representations with trivial first, second, and fourth Stiefel--Whitney classes tends to $1$ as $n \to \infty$, i.e.,
\[
\lim_{n\to\infty}
\frac{
\#\{\pi \in \op{OIrr}(G_n) : w_1(\pi)=w_2(\pi)=w_4(\pi)=0\}
}{
|\op{OIrr}(G_n)|
}
= 1.
\]
\end{enumerate}
\end{lthm}
\noindent Thus, in the large rank limit, almost all orthogonal representations are spinorial, and in fact exhibit vanishing of higher characteristic classes as well. 
\par In contrast to the large rank regime, we also investigate the complementary setting in which the rank is fixed and the size of the finite field grows. Focusing on $\GL_2(\F_q)$ as $q \to \infty$, we show that the behavior of character values, and hence of the associated Stiefel--Whitney classes, differs markedly from the large $n$ case. 

In contrast to the large rank regime, the situation for $\GL_2(\F_q)$ as $q \to \infty$ is governed by the rigid structure of its representation theory. The set of irreducible representations decomposes into four natural families: one-dimensional representations, principal series representations, twists of the Steinberg representation, and cuspidal representations arising from characters of $\F_{q^2}^\times$. From the point of view of Stiefel--Whitney classes, the relevant subset is the collection $\mathcal{O}_q$ of irreducible orthogonal representations, i.e.\ those which are self-dual. Our main result shows that the vanishing of the second Stiefel--Whitney class occurs with a limiting probability of $3/8$.

\begin{lthm}[Theorem \ref{main thm 2 corrected}]
Let $G_q=\GL_2(\F_q)$ with $q$ odd, and let $\mathcal{O}_q$ denote the set of irreducible orthogonal complex representations of $G_q$. Then
\[
\lim_{\substack{q\to\infty\\ q\equiv a\pmod{8}}}
\frac{\#\{\pi\in \mathcal{O}_q : w_2(\pi)=0\}}{|\mathcal{O}_q|}
=
\begin{cases}
1 & \text{if } a=1,\\[4pt]
\frac14 & \text{if } a=3,5,\\[4pt]
0 & \text{if } a=7.
\end{cases}
\]

Furthermore,
\[
\lim_{X\rightarrow \infty}
\frac{\sum_{\substack{q\leq X\\
q\text{ is odd}}}\#\{\pi\in \mathcal{O}_q : w_2(\pi)=0\}}
{\sum_{\substack{q\leq X\\
q\text{ is odd}}}|\mathcal{O}_q|}
=
\frac{3}{8},
\]where in the sums above, $q$ ranges over odd prime powers.
\end{lthm}

\par Our asymptotic divisibility results are closely related to recent progress on the behavior of character values in large families of finite groups, particularly for symmetric groups. In a striking result, Peluse and Soundararajan \cite{PeluseSound} show that, for any fixed prime $\ell$, almost every entry in the character table of $S_n$ is divisible by $\ell$ as $n \to \infty$, confirming a conjecture of Miller. 
\par The structure of the paper is as follows. In \S \ref{s 2}, we recall the necessary background on Stiefel--Whitney classes and the representation theory of $\GL_n(\F_q)$, including Green's parametrization and explicit formulas for low-degree characteristic classes. In \S \ref{s 3}, we establish the asymptotic divisibility results for character values and deduce the main vanishing theorems. Finally, in \S \ref{s 4}, we analyze the large $q$ limit for $\GL_2(\F_q)$, providing a contrasting perspective to the large rank behavior.

\subsection*{Acknowledgment} The author thanks Prof.~Steven Spallone for his insightful suggestions and for answering many of the author's questions. He also thanks Prof. Hamid Usefi for pointing out an error in an earlier version of this article. 

\subsection*{Data availability} This manuscript has no associated data.
\subsection*{Conflict of interest statement} There is no conflict of interest to report.

\subsection*{Funding}
The author has no funding to report.

\section{Representation theory of finite $\op{GL}_n(\F_q)$ and Stiefel--Whitney classes}\label{s 2}
\par In this section, we recall some key definitions and properties of the \emph{Stiefel--Whitney classes} of irreducible real representations of a finite group as well as the representation theory of general linear groups over finite fields.
\subsection{Stiefel--Whitney classes of a representation}

Let $G$ be a finite group. Throughout this paper, all representations $(\pi,V)$ are finite-dimensional complex representations. We write $\op{Irr}(G)$ for the set of isomorphism classes of irreducible representations of $G$. Given a representation $(\pi, V)$, we let $(\pi^\vee,V^\vee)$ be the dual representation. A representation $\pi$ of $G$ is called \emph{orthogonal} if there exists a non-degenerate $G$-invariant symmetric bilinear form $B \colon V \times V \to \mathbb{C}$. On the other hand, $\pi$ is called \emph{symplectic} if there exists a non-degenerate $G$-invariant alternating bilinear form. If $\pi$ is irreducible and self-dual, then it is either orthogonal or symplectic.

\par Write $\chi_{\pi}(g)$ for the character of $\pi$ at an element $g \in G$. Then when $\pi$ is irreducible, set
\[\varepsilon(\pi)=\frac{1}{|G|} \sum_{g \in G} \chi_{\pi}\left(g^{2}\right).\] One has that
$$
\varepsilon(\pi)= \begin{cases}0, & \pi \text { is not self-dual }\\ 1, & \pi \text { is orthogonal } \\ -1, & \pi \text { is symplectic. }\end{cases}
$$
The invariant $\epsilon(\pi)$ is referred to as the \emph{Frobenius--Schur indicator} of $\pi$. Our main focus in this article will be the family of orthogonal representations. 

\begin{proposition}
    Let $(\pi, V)$ be a complex representation of a finite group $G$, the following are equivalent:
    \begin{enumerate}
        \item[(a)]$(\pi, V)$ is orthogonal,
        \item[(b)] there exists a representation $\left(\pi_{0}, V_{0}\right)$, with $V_{0}$ a real vector space, such that $\pi \cong \pi_{0} \otimes_{\mathbb{R}} \mathbb{C}$.
    \end{enumerate}
\end{proposition}
\begin{proof}
    For a proof of this result, see \cite[Chapter II, Proposition 6.4]{BrokerDieck}.
\end{proof}
Given a $d$-dimensional real vector bundle $E$ over a paracompact base space $B$, let $w_1(E), \dots, w_d(E)\in H^*(B, \Z/2\Z)$ be the associated Stiefel--Whitney classes (cf. \cite[\S~4]{Milnor}). Let $G$ be a finite group. Recall that there exists a classifying space $BG$
with universal principal $G$-bundle $EG \to BG$, unique up to homotopy. \begin{definition}
    Given a finite-dimensional real representation $(\pi,V)$ of $G$, one may
form the associated real vector bundle
\[
EG[V] := EG \times_G V \longrightarrow BG,
\]
where $G$ acts diagonally on $EG \times V$. The $i$th \emph{Stiefel--Whitney
class} of $\pi$ is defined by
\[
w_i(\pi) := w_i(EG[V]) \in H^i(BG;\mathbb{F}_2) \cong H^i(G;\mathbb{F}_2),
\]
and we write $w(\pi)=\sum_{i\ge0} w_i(\pi)$ for the total
Stiefel--Whitney class.
\end{definition} 
\noindent One has that $w_0(\pi)=1$ and $w_i(\pi)=0$ for
$i>\dim V$. Given an orthogonal representation $\pi$, we shall set $w_i(\pi)$ to be the $i$-th Stiefel-Whitney class of the real representation $\pi_0$. 

\par The Stiefel--Whitney classes $w_i(\pi)$ satisfy a number of key properties:
\begin{enumerate}
    \item[(a)] for any group homomorphism $\varphi:H\to G$, one has that
\[
w(\pi\circ\varphi)=\varphi^* w(\pi).
\]
\item[(b)] For orthogonal representations $\pi_1$ and $\pi_2$ of $G$, we have that
\[w(\pi_1\oplus\pi_2)=w(\pi_1)\cup w(\pi_2).\]
\item[(c)] The first Stiefel--Whitney class satisfies
$w_1(\pi)=w_1(\det\pi)$. In particular, $w_1(\pi)=0$ if and only if $\det \pi=1$. 
\item[(d)] The class $w_2(\pi)$ is a cohomological obstruction to lifting representations. The orthogonal group $\op{O}(V)$ admits a $2$-fold double cover. When $w_1(\pi)=0$, we have that $w_2(\pi)=0$ if and only if $\pi:G\to \op{O}(V)$ lifts to a homomorphism $\widetilde{\pi}:G\rightarrow\op{Pin}(V)$ as depicted:
\[
\begin{tikzpicture}[node distance = 2.5 cm, auto]
    \node at (0,0) (G) {$G$};
    \node (A) at (3,0) {$\op{O}(V)$.};
    \node (B) at (3,2) {$\op{Pin}(V)$};
    \draw[->] (G) to node [swap]{$\pi$} (A);
    \draw[->] (B) to node{} (A);
    \draw[->] (G) to node {$\widetilde{\pi}$} (B);
\end{tikzpicture}
\]

\end{enumerate}
The \emph{obstruction class} is the Stiefel--Whitney class $w_i(\pi)$ with minimal $i$ such that $w_i(\pi)\neq 0$. The obstruction class is always a power of $2$ (cf. \cite[p.~50]{Benson}).
\begin{definition}
    Let $(\pi, V)$ be an irreducible orthogonal representation for which $w_1(\pi)=0$. Then $\pi$ is said to be \emph{spinorial} if $\pi$ lifts to $\widetilde{\pi}$ as above.
\end{definition}

\subsection{Irreducible representations of $S_n$}
\par We briefly recall some relevant facts from the representation theory of the
symmetric group \(S_n\), which will be used throughout the paper.
Let \(n\) be a positive integer and let
\(\lambda=(\lambda_1,\lambda_2,\dots,\lambda_k)\) be a partition of \(n\),
with \(\lambda_1\geq \lambda_2\geq \cdots \geq \lambda_k>0\).
We write \(\lambda\vdash n\) to indicate that \(\lambda\) is a partition of
\(n\). To each partition \(\lambda\vdash n\) one associates the \emph{Specht module} $\pi_\lambda$, which is an irreducible complex
representation of \(S_n\). This correspondence $\lambda\mapsto \pi_\lambda$ gives a bijection between the set of partitions of
\(n\) and the set \(\op{Irr}(S_n)\) of irreducible complex representations of
\(S_n\). We denote by \(\chi^\lambda\) the character of \(\pi_\lambda\).
The degree of \(\pi_\lambda\) is given by the hook-length formula,
\begin{equation}\label{hooklengthformula}\mathfrak{f}_\lambda:=\dim \pi_\lambda=\frac{n!}{\prod_{(i,j)\in\lambda} h_{i,j}},
\end{equation}
where \(h_{i,j}\) denotes the hook length of the box \((i,j)\) in the Young
diagram of \(\lambda\). 

\par Let $\ell$ be a prime. For a nonzero integer $n$, we write $v_\ell(n)$ for the $\ell$-adic valuation of $n$, that is, the largest integer $r \ge 0$ such that $\ell^r \mid n$. This extends in the usual way to rational numbers by setting
\[
v_\ell\!\left(\frac{a}{b}\right) = v_\ell(a) - v_\ell(b),
\qquad a,b \in \Z,\; b \neq 0.
\]
The starting point for our investigations is the following statistical result which implies that, for a fixed prime $\ell$, the quantity $v_\ell(\mathfrak{f}_\lambda)$ grows very quickly \emph{on average} as $\lambda$ ranges over partitions of large size. Let $p(n)$ be the number of partitions of a natural number $n$.
\begin{theorem}[Ganguly, Prasad, Spallone]\label{GPS thm 1}
    For every prime number $\ell$ and $r>0$,
$$
\lim _{n \rightarrow \infty} \frac{\#\left\{\lambda \vdash n \mid v_{\ell}\left(\mathfrak{f}_{\lambda}\right)<r+\log _{\ell} n\right\}}{p(n)}=0 .
$$
\end{theorem}
\begin{proof}
    For the proof of this result, please see \cite[Theorem A]{GPS}. The proof relies on a theorem of Macdonald \cite{MacdonaldBLMS} relating $v_\ell(\mathfrak{f}_\lambda)$ to the $\ell$-core tower of $\pi_\lambda$. 
\end{proof}
\par Given functions $f,g:\Z_{\ge 1}\to\mathbb{R}_{\ge 0}$, we recall that
\[
f(n)\sim g(n)
\quad\text{if}\quad
\lim_{n\to\infty}\frac{f(n)}{g(n)}=1,
\]
and that
\[
f(n)=o(g(n))
\quad\text{if}\quad
\lim_{n\to\infty}\frac{f(n)}{g(n)}=0.
\]
These notations allow one to compare the asymptotic growth of arithmetic functions in a precise manner.
\par A fundamental result of Hardy and Ramanujan, obtained via the circle method, asserts that the partition function satisfies the asymptotic formula
\[
p(n)\sim \frac{1}{4n\sqrt{3}}\exp\!\left(\pi\sqrt{\frac{2n}{3}}\right).
\]
In this context, Theorem~\ref{GPS thm 1} asserts that for any fixed prime $\ell$ and integer $r$, the set
\[
\{\lambda \vdash n : v_\ell(\mathfrak{f}_\lambda) < r + \log_\ell n\}
\]
is negligible compared to the full set of partitions of $n$. More precisely, combining Theorem~\ref{GPS thm 1} with the asymptotic for $p(n)$, we obtain the equivalent formulation
\[
\#\{\lambda \vdash n : v_\ell(\mathfrak{f}_\lambda) < r + \log_\ell n\}
=
o\!\left(\frac{1}{4n\sqrt{3}}\exp\!\left(\pi\sqrt{\frac{2n}{3}}\right)\right).
\]
\noindent This phenomenon plays a crucial role in the asymptotic analysis of character degrees and character values, where the quantities $\mathfrak{f}_\lambda$ appear naturally as building blocks.
\par Given $k\leq n$ and a partition $\mu=(\mu_1, \dots, \mu_m)$ of $k$, denote by $\chi_{\mu}^{\lambda}$ the value of $\chi^{\lambda}$ at an element of cycle type $\left(\mu_{1}, \ldots, \mu_{m}, 1^{n-k}\right)$. Let $(n)_{k}$ denote the falling factorial
\[(n)_{k}=n(n-1) \cdots(n-k+1).\] A theorem of Lassalle \cite[Theorem 6]{Lassalle} defines a rational number $A_{\mu}^{\lambda}$ such that
\begin{equation}\label{lassalle}
\chi_{\mu}^{\lambda}=\frac{f_{\lambda}}{(n)_{k}} A_{\mu}^{\lambda}.
\end{equation}
It is shown in \cite[section 4]{GPS} that $A_{\mu}^{\lambda}$ is in fact an integer, and this leads to the following result.
\begin{theorem}[Ganguly, Prasad, Spallone]
Let $k$ and $d$ be positive integers, and let $\mu$ be a partition of $k$. Then
\[
\lim_{n \to \infty}
\frac{
\#\left\{\lambda \vdash n \; \middle| \; \chi_\mu^\lambda \text{ is divisible by } d \right\}
}{p(n)}
= 1.
\]
\end{theorem}

\begin{proof}
This is the principal result of \emph{loc.\ cit.} We sketch the argument for the convenience of the reader.
\par Fix a prime $q$. By the identity~\eqref{lassalle}, one has the lower bound
\[
v_q\!\left(\chi_\mu^\lambda\right)
\ge
v_q\!\left(\mathfrak{f}_\lambda\right)
-
v_q\!\left((n)_k\right).
\]
Applying Legendre's formula,
\[
v_q(n!) = \frac{n - a_q(n)}{q-1},
\]
where $a_q(n)$ denotes the sum of the base-$q$ digits of $n$, we obtain
\[
v_q\!\left((n)_k\right)
=
v_q\!\left(\frac{n!}{(n-k)!}\right)
=
\frac{k + a_q(n-k) - a_q(n)}{q-1}.
\]
Since $a_q(n) \le (q-1)\log_q n$, it follows that
\[
v_q\!\left((n)_k\right)
\le
k + (q-1)\log_q n.
\]
Consequently, if
\[
v_q\!\left(\mathfrak{f}_\lambda\right)
\ge
m + (q-1)\log_q n,
\]
then
\[
v_q\!\left(\chi_\mu^\lambda\right)
\ge
m - k.
\]
Taking $m = k + b$ and invoking Theorem~\ref{GPS thm 1}, we deduce that
\[
\lim_{n \to \infty}
\frac{
\#\left\{\lambda \vdash n \; \middle| \; v_q\!\left(\chi_\mu^\lambda\right) \le b \right\}
}{p(n)}
=
0.
\]
Since this holds for each prime divisor $q$ of $d$, the stated result follows.
\end{proof}
\begin{proposition}
    The irreducible representations $\pi_\lambda$ of $S_n$ are all realizable over $\mathbb{R}$ (in fact over $\Q$).
\end{proposition}
\begin{proof}
    This result is well known and follows from how they are defined using Young tableaux. 
\end{proof}
\noindent As a consequence, all irreducible representations $\pi_\lambda$ of $S_n$ are orthogonal and the Stiefel Whitney classes $w_i(\pi)$ are defined.

\par Let $(\pi, V)$ be a complex representation of the symmetric group $S_{n}$, then \[\operatorname{det} \rho: S_{n} \rightarrow \mathbb{C}^{*}\] is either the trivial character or the sign character. One says that $(\pi, V)$ is a \emph{chiral representation} if $\operatorname{det} \rho$ is the sign character of $S_{n}$ and \emph{achiral otherwise}. Let $b(n)$ be the number of chiral irreducible representations of $S_n$.
\begin{theorem}[Ayyer, Prasad, Spallone]
     If $n$ is an integer having binary expansion
\begin{equation*}
n=\epsilon+2^{k_{1}}+2^{k_{2}}+\cdots+2^{k_{r}}, \quad \text{with}\quad \epsilon \in\{0,1\}, \quad 0<k_{1}<k_{2}<\cdots<k_{r},
\end{equation*}
then 
\[b(n)=2^{k_{2}+\cdots+k_{r}}\left(2^{k_{1}-1}+\sum_{v=1}^{k_{1}-1} 2^{(v+1)\left(k_{1}-2\right)-\binom{v}{2}}+\epsilon 2^{\binom{k_{1}}{2}}\right).\]
\end{theorem}
\begin{proof}
    This result is \cite[Theorem 1]{APS}. 
\end{proof}
\par We record an immediate consequence of the general philosophy that low-degree characteristic classes become asymptotically trivial in large rank. In the case of the symmetric group, this phenomenon already appears in a particularly transparent form.
\par Let $\pi_\lambda$ denote the irreducible representation of the symmetric group $S_n$ corresponding to a partition $\lambda \vdash n$, realized over $\mathbb R$. Recall that the first Stiefel--Whitney class $w_1(\pi_\lambda)$ is the reduction modulo $2$ of the determinant character. In particular, the condition $w_1(\pi_\lambda)=0$ is equivalent to the statement that $\pi_\lambda$ is \emph{achiral}, i.e.\ has trivial determinant.

\begin{theorem}[Ayyer, Prasad, Spallone]
Given a natural number $n$, let $p(n)$ denote the number of partitions of $n$. Then
\[
\lim_{n\to\infty}
\frac{\#\{\lambda \vdash n : \det \pi_\lambda = 1\}}{p(n)} = 1.
\]
\end{theorem}

\noindent
In other words, asymptotically almost all Specht modules are achiral. Equivalently, the obstruction class $w_1(\pi_\lambda)$ vanishes for $100\%$ of partitions as $n \to \infty$.

\par

A deeper refinement of this statement is obtained by considering the second Stiefel--Whitney class. Using an explicit formula expressing $w_2(\pi_\lambda)$ in terms of character values on elements of order $2$, Ganguly and Spallone \cite{GangulySpallone} show that $w_2=0$ with probability $1$.

\begin{theorem}[Ganguly, Spallone]
One has
\[
\lim_{n\to\infty}
\frac{
\#\{\lambda \vdash n : w_1(\pi_\lambda)=0 \ \text{and} \ w_2(\pi_\lambda)=0\}
}{p(n)}
= 1.
\]
\end{theorem}

\noindent
Thus, not only are almost all Specht modules achiral, but they are in fact \emph{spinorial}, in the sense that both $w_1$ and $w_2$ vanish. From a topological perspective, this means that the associated real representations admit lifts to the spin group with probability tending to $1$ as $n\to\infty$.

\par These results provide a guiding analogy for the situation of finite groups of Lie type considered in this paper. In both settings, the key input is an explicit expression for low-degree Stiefel--Whitney classes in terms of character values at elements of order $2$, combined with asymptotic divisibility properties of these character values.
\subsection{Green's parametrization and duality for $\op{GL}_n(\F_q)$}

We fix a finite field $k\simeq \F_q$ and an algebraic closure $\bar{k}$. Denote by $\op{G}_k$ the absolute Galois group $\op{Gal}(\bar{k}/k)$. Let $n\geq 2$ be a natural number. In this section we recall the parametrization of the irreducible complex
representations of $G_n := \op{GL}_n(k)$ due to Green \cite{Green}. Let $V$ be an $n$-dimensional $k$ vector space which is also a $k$-variety and identify $G_n$ with $\op{Aut}_k(V_k)$. We give a self-contained account sufficient for our
purposes, and we prove the precise behavior of irreducible representations
under contragredient duality.
\par We recall the parametrization of conjugacy classes in $G_n$.  Let $f(t)=t^d-a_{d-1}t^{d-1}-\cdots-a_0\in k[t]$
be a monic polynomial of degree $d$.  Define the $d\times d$ matrix
\[
U(f)=
\begin{pmatrix}
0 & 1 &  &  & \\
  & 0 & 1 &  & \\
  &   & \ddots & \ddots & \\
  &   &        & 0 & 1\\
a_0 & a_1 & \cdots & a_{d-2} & a_{d-1}
\end{pmatrix}.
\]
For an integer $m\ge1$ let $U_m(f)$ be the block matrix
\[
U_m(f)=
\begin{pmatrix}
U(f) & I_d &  &  \\
     & U(f) & I_d &  \\
     &      & \ddots & \ddots \\
     &      &        & U(f)
\end{pmatrix},
\]
with $m$ diagonal blocks $U(f)$ and identity blocks $I_d$ on the
superdiagonal.  If $\lambda=(\lambda_1,\dots,\lambda_r)$ is a partition
of a positive integer $k$, we define
\[
U_\lambda(f)=\operatorname{diag}\bigl(U_{\lambda_1}(f),\dots,U_{\lambda_r}(f)\bigr).
\]
The characteristic polynomial of $U_\lambda(f)$ is $f(t)^k$.

\par Let $\Phi$ be the set of irreducible monic polynomials in $f(T)\in k[T]$ such that $f(T)\neq T$. One may identify $\Phi$ with the set of orbits of $\op{G}_k$ acting on $\bar{k}^\times$. We write $d(f)$ for the degree of a polynomial $f$. The
conjugacy classes in $G_n$ are parametrized by functions $\mu \in \mathfrak{X}_n$. If $\mu(f_i)=\nu_i$, the corresponding conjugacy class will be denoted
\[
c=(f_1^{\nu_1}f_2^{\nu_2}\cdots f_N^{\nu_N}).
\]
\par For $f\in\Phi$ define the reciprocal polynomial
\[
f^*(T)=T^{\deg f}f(1/T).
\]
Let
\[
\mathfrak X_n=\left\{\mu:\Phi\to\Lambda\;\middle|\;\sum_{f\in\Phi}(\deg f)|\mu(f)|=n\right\},
\]
where $\Lambda$ denotes the set of partitions and $|\lambda|$ the size of a partition $\lambda$. There is a natural bijection between the set of conjugacy classes of $G_n$ and $\mathfrak X_n$, see \cite[Lemma 2.3]{Springer} for further details. For $\mu\in\mathfrak X_n$ let $\chi_\mu$ denote the irreducible character of $G_n$ associated to $\mu$ in Green's parametrization. Define $\mu'\in\mathfrak X_n$ by
\[
\mu'(f)=\mu(f^*).
\]

\par Let $n=s_1+s_2+\cdots+s_k$
be a partition of $n$ into positive integers and let $V_n$ be the
$n$--dimensional vector space over $k$.  For $1\le i\le k$ let $V^{(i)}$
be the subspace of $V_n$ consisting of vectors whose first
$s_1+\cdots+s_i$ coordinates are zero.  This gives a chain of subspaces
\[
V^{(0)} \supset  V^{(1)} \supset \cdots \supset V^{(k)} = 0 .
\]

Let $\mathfrak{H}_{s_1,\dots,s_k}$ be the subgroup of $G_n$ consisting of
all elements which leave this chain invariant.  In matrix form these are
the block upper triangular matrices
\[
A=
\begin{pmatrix}
A_{11} & A_{12} & \cdots & A_{1k}\\
0 & A_{22} & \cdots & A_{2k}\\
\vdots & \vdots & \ddots & \vdots\\
0 & 0 & \cdots & A_{kk}
\end{pmatrix},
\qquad
A_{ii}\in G_{s_i}.
\]

If $\alpha_i$ is a class function (in particular a character) of
$G_{s_i}$ for $1\le i\le k$, we define a class function $\psi$ on
$\mathfrak{H}_{s_1,\dots,s_k}$ by
\[
\psi(A)=\alpha_1(A_{11})\cdots\alpha_k(A_{kk}).
\]
The \emph{$\circ$-product}
\[
\alpha_1\circ\alpha_2\circ\cdots\circ\alpha_k
\]
is defined to be the character of $G_n$ obtained by inducing $\psi$ from
$\mathfrak{H}_{s_1,\dots,s_k}$ to $G_n$.  As shown in \cite[\S2]{Green},
the binary operation $\circ$ is commutative, associative, and bilinear.
If $\mathcal A_n$ denotes the space of class functions on $G_n$ and
\[
\mathcal A=\bigoplus_{n\ge1}\mathcal A_n,
\]
then $\mathcal A$ becomes a commutative associative algebra over
$\mathbb{C}$ under the product $\circ$.
\begin{proposition}
Let $\chi_\mu$ denote the irreducible character of $G_n$ corresponding
to $\mu$ in Green's parametrization. With respect to notation above, we have that
\[
(\chi_\mu)^\vee=\chi_{\mu'}.
\]
In particular $\chi_\mu$ is self-dual if and only if
\[
\mu(f)=\mu(f^*)
\qquad\text{for all }f\in\Phi .
\]
\end{proposition}

\begin{proof}
The result is probably well known and follows from the construction of Deligne and Lusztig. Since we have been unable to find a reference, we give a brief sketch of a proof here following the notation and conventions in \cite{Green}. 
\par Let $g\in G_n$ and let
\[
\alpha_1,\dots,\alpha_n\in\overline{\F}_q^\times
\]
be the eigenvalues of $g$. Then the eigenvalues of $g^{-1}$ are
$\alpha_1^{-1},\dots,\alpha_n^{-1}$. From Green's definition
\cite[Definition~3.1]{Green} of the functions
\[
T_{1,d}(k;\alpha)
=
\theta_d^{k}(\alpha)
+
\theta_d^{kq}(\alpha)
+\cdots+
\theta_d^{kq^{d-1}}(\alpha),
\]
one immediately obtains
\[
T_{1,d}(k;\alpha^{-1})=T_{1,d}(-k;\alpha).
\]
Consequently the class functions $J_d(k)$ defined by
\[
J_d(k)(g)
=
\sum T_{1,d}(k;\alpha_{i_1})\cdots T_{1,d}(k;\alpha_{i_d})
\]
satisfy
\[
J_d(k)(g^{-1})=J_d(-k)(g).
\]
Since the basic characters $B^\rho(h^\rho;\cdot)$ are obtained from the
functions $J_d(k)$ by Green's induction product, it follows that
\[
B^\rho(h^\rho;g^{-1})=B^\rho((-h)^\rho;g),
\]
i.e. inversion replaces each parameter $h_{d,i}$ by $-h_{d,i}$.

Theorem~14 in \emph{loc. cit.} expresses the symmetric--function class function
$(\cdots g^{\mu(g)}\cdots)$ as a linear combination of the basic
characters $B^\rho(h^\rho;\cdot)$. Replacing $\mu$ by
$\mu'$ with $\mu'(f)=\mu(f^*)$ corresponds to replacing every root
$\alpha$ of $f$ by $\alpha^{-1}$, hence to replacing the parameters
$h$ by $-h$ in the expansion. Comparing the resulting formulas shows
that
\[
\chi_\mu(g^{-1})=\chi_{\mu'}(g)
\qquad\text{for all }g\in G_n.
\]
Since the character of the dual representation satisfies
$\chi_{V^\vee}(g)=\chi_V(g^{-1})$, this implies
\[
(\chi_\mu)^\vee=\chi_{\mu'}.
\]
The final assertion is immediate from the definition of $\mu'$.
\end{proof}
\begin{proposition}\label{selfdualimpliesorthogonal}
With the above notation, the irreducible character $\chi_\mu$ is orthogonal if and only if it is self-dual.
\end{proposition}

\begin{proof}
By \cite[Theorem 4]{DPrasad}, an irreducible representation of $\op{GL}_n(\F_q)$ is orthogonal if and only if it is self-dual and has Frobenius--Schur indicator $+1$. For irreducible characters of $\op{GL}_n(\F_q)$, self-duality already forces the indicator to be $+1$, and the claim follows.
\end{proof}

\subsection{Stiefel--Whitney classes for representations of $\op{GL}_n(\F_q)$}

\par We recall results from \cite{JoshiSpallone1,JoshiSpallone2, GangulyJoshi2,GangulyJoshiTotalSWCs}, where Stiefel--Whitney classes of real representations of $\op{GL}_n(\F_q)$ are computed in terms of character values on elements of order $2$.

Let $\Omega = \{0,1\}^n$ be the set of binary strings of length $n$. Let $C_m$ be a cyclic group of order $m$ with generator $x$. For $j \in \Z/m\Z$, define the linear character
\[
\chi^j(x) = \zeta_m^{jx},
\]
where $\zeta_m$ is a fixed primitive $m$th root of unity. If $m$ is even, we define the sign character $\op{sgn} := \chi^{m/2}$. For $a =(a_1, \dots, a_n)\in \Omega$, define a representation of $C_m^n$ by
\[
\op{sgn}_a := \boxtimes_{j=1}^n (\op{sgn})^{a_j}.
\]We identify $C_2^n$ with the subgroup of diagonal matrices
\[
\op{diag}(\pm1,\dots,\pm1) \subset \op{GL}_n(\F_q).
\]
For $0\le i \le n$, define
\[
h_i := \op{diag}(\underbrace{-1,\dots,-1}_{i},1,\dots,1).
\]
Let $x$ be an element of $\F_q^\times$ which is not a square and let \[t_x:=\op{diag}(x, 1, \dots, 1)\in \op{GL}_n(\F_q).\] When $q\equiv 3\pmod{4}$ we can take $x:=-1$, in which case $t_x=h_1$. Given a finite dimensional real representation $\pi$ of $\op{GL}_n(\F_q)$ with character $\chi_\pi$, we define
\[
m_\pi = \frac{\dim \pi - \chi_\pi(h_1)}{2}\quad \text{and}\quad m_x:=\frac{\dim \pi - \chi_\pi(t_x)}{2}.
\]
\par We recall a standard relation between the determinant and the character of a real representation evaluated at an involution.

\begin{lemma}\label{lemma det involution}
Let $\pi$ be a finite-dimensional real representation of a finite group $G$, with character $\chi_\pi$. Let $g \in G$ be an element of order dividing $2$. Then
\[
\det(\pi)(g)
=
(-1)^{\frac{\dim \pi - \chi_\pi(g)}{2}}.
\]
\end{lemma}

\begin{proof}
Since $g^2=1$, the linear operator $\pi(g)$ satisfies $\pi(g)^2=I$. It follows that $\pi(g)$ is diagonalizable over $\mathbb R$, with eigenvalues contained in $\{\pm 1\}$. Let $V$ be the underlying real representation space of $\pi$, and write
\[
V = V^{+} \oplus V^{-}
\]
for the decomposition into the $+1$ and $-1$ eigenspaces of $\pi(g)$. Set $d_+ = \dim V^{+}$ and $d_- = \dim V^{-}$, then \[\dim \pi = d_+ + d_-.\] On the other hand, the character value is given by the trace:
\[
\chi_\pi(g) = \operatorname{tr}(\pi(g)) = d_+ - d_-.
\]
Solving these two equations for $d_-$, we obtain
\[
d_- = \frac{\dim \pi - \chi_\pi(g)}{2}.
\]
We thus find that
\[
\det(\pi)(g)
=
(-1)^{d_-}
=
(-1)^{\frac{\dim \pi - \chi_\pi(g)}{2}},
\]
as claimed.
\end{proof}
\par Applying Lemma~\ref{lemma det involution} to the element $t_x \in \op{GL}_n(\F_q)$, we obtain
\[
\det(\pi)(t_x)
=
(-1)^{\frac{\dim \pi - \chi_\pi(t_x)}{2}}
=
(-1)^{m_x}.
\]
For a real representation $\pi:G\to \mathrm{O}(V)$, the first Stiefel--Whitney class is the homomorphism
\[
w_1(\pi) = \det  \pi : G \longrightarrow \{\pm 1\}.
\]
In particular, $w_1(\pi)$ is a group homomorphism, and hence factors through the abelianization $G^{\mathrm{ab}}=G/[G,G]$.

\par In the case $G=G_n=\GL_n(\F_q)$, the determinant map
\[
\det : G_n \longrightarrow \F_q^\times
\]
induces an isomorphism
\[
G_n^{\mathrm{ab}} \;\cong\; \F_q^\times.
\]
Composing with the natural quotient $\F_q^\times \to \F_q^\times/(\F_q^\times)^2 \cong \{\pm1\}$, we see that every homomorphism $G_n \to \{\pm1\}$ factors through the determinant modulo squares. Consequently, such a homomorphism is completely determined by its value on any element whose determinant represents the nontrivial class in $\F_q^\times/(\F_q^\times)^2$. The element $t_x$
has determinant $x$, which is a nonsquare in $\F_q^\times$ and represents the unique nontrivial element of $\F_q^\times/(\F_q^\times)^2$. It follows that a homomorphism $G_n \to \{\pm1\}$ is trivial if and only if it takes the value $1$ at $t_x$. In particular,
\begin{equation}\label{w_1 basic fact}w_1(\pi)=0
\quad\Longleftrightarrow\quad m_x\, \text{is even.}
\end{equation}

\par Let $C_2^n \subset \op{GL}_n(\F_q)$ be the subgroup of the diagonal subgroup consisting of entries $\pm 1$. For $1 \le i \le n$, define
\[
v_i = w_1(\op{sgn}_{e_i}) \in H^1(C_2^n;\Z/2\Z),
\]
where $e_i$ is the $i$th standard basis vector. Let
\[
D = \left\{ \diag(d_1,\dots,d_n) : d_i \in \F_q^\times \right\} \subset \GL_n(\F_q)
\]
denote the diagonal torus. For each $1 \le i \le n$, let
\[
\chi_i : D \to \F_q^\times, \qquad \chi_i(\diag(d_1,\dots,d_n)) = d_i
\]
be the projection onto the $i$th coordinate. Composing $\chi_i$ with a fixed embedding $\F_q^\times \hookrightarrow \mathbb{C}^\times$, we obtain a one-dimensional complex representation of $D$, which we continue to denote by $\chi_i$.

We write $\chi_{i,\mathbb{R}}$ for the underlying real representation. We define
\[
t_i := w_2(\chi_{i,\mathbb{R}}) \in H^2(D;\Z/2\Z).
\]
Let $C_2^n \subset \GL_n(\F_q)$ denote the subgroup of diagonal matrices with entries $\pm 1$, and for $1 \le i \le n$ let
\[
v_i = w_1(\mathrm{sgn}_{e_i}) \in H^1(C_2^n;\Z/2\Z)
\]
be as above. As is well known (cf. \cite[Theorem 5]{GangulyJoshiTotalSWCs}) the diagonal subgroup $D$ detects the $\Z/2\Z$-cohomology of $\op{GL}_n(\F_q)$, i.e., the restriction map $w\mapsto w_{|D}$ from 
\[H^*(\op{GL}_n(\F_q), \Z/2\Z)\rightarrow H^*(D, \Z/2\Z)\]is injective. If $q\equiv 3\pmod{4}$ then $C_2^n$ is the Sylow $2$-subgroup of $D$ and thus in this case, the restriction 
\[H^*(\op{GL}_n(\F_q), \Z/2\Z)\rightarrow H^*(C_2^n, \Z/2\Z)\]is injective.

\begin{theorem}\label{theorem about w_2 and w_4}
Let $\pi$ be a finite-dimensional real representation of $\GL_n(\F_q)$.

\begin{enumerate}
\item If $q \equiv 1 \pmod{4}$, then upon restriction to the diagonal torus $D$ one has
\[
w_2(\pi)\big|_D = \frac{m_\pi}{2} \sum_{i=1}^n t_i \in H^2(D;\Z/2\Z),
\]
and
\[
w_4(\pi)\big|_D
=
\binom{m_\pi/2}{2}\sum_{i=1}^n t_i^2
+
\frac{\dim \pi - \chi_\pi(h_2)}{8}
\sum_{1 \le i < j \le n} t_i t_j
\in H^4(D;\Z/2\Z).
\]
Further if $n\geq 5$ and $\pi$ is a principal series representation, then $w_4(\pi)=0$.
\item If $q \equiv 3 \pmod{4}$, then upon restriction to $C_2^n$ one has
\[
w_2(\pi)\big|_{C_2^n} = \binom{m_\pi}{2} \sum_{i=1}^n v_i^2 \in H^2(C_2^n;\Z/2\Z).
\]
\end{enumerate}
\end{theorem}
\begin{proof}
    The result follows from \cite[Theorem 2, the formulas on p.~4 and Theorem 4]{GangulyJoshiTotalSWCs}.
\end{proof}

\par The formulas obtained above show that $w_2(\pi)$, and (when $q \equiv 1 \pmod{4}$) $w_4(\pi)$ are determined by congruence conditions on the integers $\chi_\pi(h_i)$ for $i=0,1,2$. 

\section{Asymptotic divisibility in the large $n$ limit}\label{s 3}
\par In this section we investigate the $2$-adic divisibility properties of character degrees and character values for irreducible representations of $G_n=\GL_n(\F_q)$ in the limit as $n\to\infty$. The guiding principle is that $2$-adic valuations of character values become increasingly large on average as $n\rightarrow \infty$. More precisely, we show that for any fixed element $g \in G_{n_0}$, the values $\chi(g)$ are divisible by arbitrarily large powers of $2$ for almost all irreducible characters $\chi$ of $G_n$ as $n \to \infty$. These asymptotic divisibility results will serve as the key input in the next section, where we relate them to the vanishing of the Stiefel--Whitney classes $w_1$, $w_2$ and $w_4$.
\par Fix an odd prime power $q$. For each $n$, set $G_n := \op{GL}_n(\F_q)$. We study asymptotic properties of irreducible characters of $G_n$ as $n \to \infty$. Given $m < n$, we embed $G_m$ into $G_n$ via block diagonal matrices:
\[
g \mapsto 
\begin{pmatrix}
g & 0 \\
0 & I_{n-m}
\end{pmatrix}.
\]
If $\chi$ is a character of $G_n$ and $g \in G_m$, we define
\[
\chi(g) := \chi\left(
\begin{pmatrix}
g & 0 \\
0 & I_{n-m}
\end{pmatrix}
\right).
\]
Let $\op{Irr}(G_n)$ (resp. $\op{OIrr}(G_n)$) denote the set of irreducible characters (resp. irreducible orthogonal characters) of $G_n$. The main goal in this section is to prove the following result.

\begin{theorem}
Fix $n_0 \ge 1$, $g \in G_{n_0}$, and $d = 2^m$. Then
\[
\lim_{n\to\infty}
\frac{
\#\{ \chi \in \op{OIrr}(G_n) \mid \quad d \, \text{divides}\,\chi(g)\}
}{
|\op{OIrr}(G_n)|
}
= 1.
\]
\end{theorem}
\par Let $\Lambda$ denote the set of partitions, and $\Lambda_n$ those of size $n$. Write a partition $\lambda=\lambda_1 \geq \cdots \geq \lambda_k > 0$ as a weakly decreasing sequence of positive integers. We write $|\lambda|=\sum_i \lambda_i$ for the size of $\lambda$, and $\ell(\lambda)=k$ for its length. Let $p(n)=|\Lambda_n|$ denote the number of partitions of $n$, and write $\mathcal{H}(\lambda)$ for the set of hooks of $\lambda$. For a partition $\lambda$, define
\[
\alpha(\lambda)=\sum_i (i-1)\lambda_i.
\]

Let $\Phi$ denote the set of monic irreducible polynomials in $\F_q[x]$ different from $x$. Recall from the previous section that $\mathfrak{X}_n$ is the set of functions
\[
\boldsymbol{\mu} : \Phi \to \Lambda
\quad\text{such that}\quad
\sum_{f \in \Phi} |\boldsymbol{\mu}(f)| d(f) = n.
\]
Each $\boldsymbol{\mu}\in \mathfrak{X}_n$ corresponds to an irreducible complex character $\chi_{\boldsymbol{\mu}}$ of $G_n$. The degree $d_{\boldsymbol{\mu}}$ of $\chi_{\boldsymbol{\mu}}$ is given by
\[
d_{\boldsymbol{\mu}}
=
\psi_n(q)\prod_{f \in \Phi} H(\boldsymbol{\mu}(f), q^{d(f)}),
\]
where
\[
\psi_n(q) := \prod_{i=1}^n (q^i - 1),
\quad\text{and}\quad 
H(\lambda,x): = x^{\alpha(\lambda)} \prod_{h \in \mathcal{H}(\lambda)} (x^{|h|}-1)^{-1}.
\]
\begin{definition}
    Let $\mathfrak{Y}_n$ be the subset of $\mathfrak{X}_n$ consisting of all partition valued functions $\boldsymbol{\mu}$ for which 
    \[\boldsymbol{\mu}(f)=\boldsymbol{\mu} (f^*)\quad \text{for all}\quad f\in \Phi.\]
\end{definition}
\noindent It follows from Proposition \ref{selfdualimpliesorthogonal} that there is a bijection between $\mathfrak{Y}_n$ are the irreducible orthogonal representations of $G_n$.
\par For each partition $\lambda \vdash n$, define $\boldsymbol{\mu}_\lambda \in \mathfrak{X}_n$ by
\[
\boldsymbol{\mu}_\lambda(x-1)=\lambda,
\qquad
\boldsymbol{\mu}_\lambda(f)=\emptyset \quad \text{for } f(x)\neq x-1.
\]
We write $\chi_\lambda := \chi_{\boldsymbol{\mu}_\lambda}$ and call such characters \emph{unipotent}. These are precisely the irreducible constituents of the permutation representation of $G_n$ on $G_n/B_n$, where $B_n$ denotes the subgroup of upper triangular matrices.

The degree $d_\lambda = \deg(\chi_\lambda)$ is given explicitly by
\begin{equation}\label{eq:unipotent-degree}
d_\lambda(q)
=
q^{\alpha(\lambda)}
\frac{\prod_{i=1}^n (q^i - 1)}{\prod_{h \in \mathcal{H}(\lambda)} (q^{|h|}-1)}.
\end{equation}
The dimension of $\chi_{\boldsymbol{\mu}}$ factorizes as
\[
d_{\boldsymbol{\mu}} = a_{\boldsymbol{\mu}} b_{\boldsymbol{\mu}},
\]
where
\[
a_{\boldsymbol{\mu}} =
\frac{\prod_{i=1}^n (q^i-1)}{\prod_{f \in \Phi} \prod_{i=1}^{|\boldsymbol{\mu}(f)|} (q^{d(f)i}-1)},
\quad\text{and}\quad
b_{\boldsymbol{\mu}} = \prod_{f \in \Phi} d_{\boldsymbol{\mu}(f)}(q^{d(f)}),
\]
(cf. \cite[2.3.3]{shah2024divisibility}).
\begin{lemma}\label{valuation lemma}
Let $G$ be a finite group, $g \in G$, $\chi$ an irreducible character of $G$ and let $Z_G(g)$ be the centralizer of $G$. Then
\[
\frac{\chi(g)}{\deg \chi}[G:Z_G(g)]
\]
is an algebraic integer.
\end{lemma}
\begin{proof}
    The above result is \cite[Exercise 6.9]{Serre77}.
\end{proof}
\noindent When $G=G_n$ and $g\in G_{n_0}$, we find that $Z_G(g)$ contains $G_{n-n_0}$. This implies that
\[
\frac{\chi_{\boldsymbol{\mu}}(g)}{d_{\boldsymbol{\mu}}}[G_n:G_{n-n_0}]\in \bar{\Z},
\]
from which one has the following valuation criterion.

\begin{proposition}
Let $d$ be coprime to $q$. If for every prime $\ell \mid d$,
\[
v_\ell(d_{\boldsymbol{\mu}}) - v_\ell\left(\prod_{i=0}^{n_0-1} (q^{n-i}-1)\right)
\ge v_\ell(d),
\]
then $d \mid \chi_{\boldsymbol{\mu}}(g)$, i.e., $\chi_{\boldsymbol{\mu}}(g)/d\in \bar{\Z}$.
\end{proposition}
\begin{proof}
    For further details, see \cite[Corollary 3.3.1]{shah2024divisibility}.
\end{proof}

\begin{theorem}
Let $\ell$ be a prime number. We have
\[
v_\ell(d_{\boldsymbol{\mu}})
\ge
v_\ell \left(\frac{n!}{\prod_f |\boldsymbol{\mu}(f)|!}\right)
+
\sum_f v_\ell(\mathfrak{f}_{\boldsymbol{\mu}(f)}),
\]
where $f_\lambda$ is the dimension of the Specht module associated to $\lambda$ given by the hook length formula \eqref{hooklengthformula}.
\end{theorem}
\begin{proof}
    This result is \cite[Theorem 4.4]{shah2024divisibility}.
\end{proof}

\begin{definition}
    Let $\mathscr{G}_n(\Phi)$ be the set of functions 
    \[F:\Phi\rightarrow \Z_{\geq 0}\] such that 
    \begin{itemize}
        \item $\sum_{f\in \Phi} d(f)F(f)=n$, 
        \item $F(f^*)=F(f)$ for all $f\in \Phi$. 
    \end{itemize}
\end{definition}
\noindent
Thus $F$ records the multiplicities of the elements of $\Phi$ subject to the self-duality condition, and can be viewed as the “type” of an element $\boldsymbol{\mu}\in \mathfrak{Y}_n$. For a fixed $F$, we consider the fiber
\[
\{\boldsymbol{\mu}\in \mathfrak{Y}_n\mid \boldsymbol{\mu}\mapsto F\},
\]
consisting of all objects with this prescribed data. The next result asserts that for large $n$, almost all $\boldsymbol{\mu}$ in a fixed fiber have valuation at least $v((n)_k)$.
\begin{lemma}\label{mu to F lemma}
    Let $k$ be a positive integer. Then for any $\epsilon>0$ there exists a large positive integer $N=N(\epsilon)>0$ such that for all $n\geq N$ and all $F\in \mathscr{G}_n(\Phi)$ we have
    \[\frac{\# \{\boldsymbol{\mu}\in \mathfrak{Y}_n\mid \boldsymbol{\mu}\mapsto F\quad \text{and}\quad v_2(d_{\boldsymbol{\mu}})< v_2((n)_k)\}}{\# \{\boldsymbol{\mu}\in \mathfrak{Y}_n\mid \boldsymbol{\mu}\mapsto F\}}<\epsilon\]
\end{lemma}

\begin{proof}
Fix $\varepsilon>0$. By Theorem~\ref{GPS thm 1}, there exists $M=M(\varepsilon)$ such that for all $m \ge M$,
\[
\frac{\#\{\lambda \vdash m : v_2(\mathfrak{f}_\lambda) < k + \log_2(k q^{k+1}) + \log_2 m\}}{p(m)} < \varepsilon.
\]
\noindent Set $N = k q^{k+1} M$, and let $n \ge N$. Fix $F \in \mathscr{G}_n(\Phi)$, so that $F(f)=F(f^*)$ for all $f \in \Phi$. We consider the set
\[
\{\boldsymbol{\mu} \in \mathfrak{Y}_n : \boldsymbol{\mu} \mapsto F\}.
\]
Such $\boldsymbol{\mu}$ are determined by choosing, for each orbit $\{f,f^*\}$, a partition $\lambda_f \vdash F(f)$, and setting $\boldsymbol{\mu}(f)=\boldsymbol{\mu}(f^*)=\lambda_f$. Thus there is a bijection
\[
\{\boldsymbol{\mu} \in \mathfrak{Y}_n : \boldsymbol{\mu} \mapsto F\}
\;\cong\;
\prod_{\{f,f^*\}} \Lambda_{F(f)}.
\]

We now consider two cases. First suppose that $\max_{f \in \operatorname{supp}(F)} d(f) \le k$. Since there are at most $q^{k+1}$ monic polynomials of degree at most $k$, the number of orbits $\{f,f^*\}$ in $\operatorname{supp}(F)$ is at most $q^{k+1}$. Therefore,
\[
\frac{1}{|\operatorname{supp}(F)|} \sum_{f \in \Phi} F(f)
\;\ge\;
\frac{\sum_{f \in \Phi} d(f) F(f)}{k q^{k+1}}
=
\frac{n}{k q^{k+1}}
\;\ge\;
M.
\]
Hence there exists $f_0 \in \operatorname{supp}(F)$ such that $m := F(f_0) \ge M$.

For any $\boldsymbol{\mu} \mapsto F$, we have
\[
v_2(d_{\boldsymbol{\mu}})
\ge
v_2\bigl(\mathfrak{f}_{\boldsymbol{\mu}(f_0)}\bigr),
\]
since $d_{\boldsymbol{\mu}}$ contains the factor $d_{\boldsymbol{\mu}(f_0)}(q^{d(f_0)})$, whose $2$-adic valuation is bounded below by that of $\mathfrak{f}_{\boldsymbol{\mu}(f_0)}$.

It follows that
\[
\begin{aligned}
&\frac{\# \{\boldsymbol{\mu}\in \mathfrak{Y}_n : \boldsymbol{\mu}\mapsto F,\ v_2(d_{\boldsymbol{\mu}})< v_2((n)_k)\}}
{\# \{\boldsymbol{\mu}\in \mathfrak{Y}_n : \boldsymbol{\mu}\mapsto F\}} \\
&\qquad\le
\frac{\#\{\lambda \vdash m : v_2(\mathfrak{f}_\lambda) < v_2((n)_k)\}}{p(m)}.
\end{aligned}
\]

Now observe that
\[
v_2((n)_k)
=
\sum_{i=0}^{k-1} v_2(n-i)
\;\ge\;
\log_2(n-k+1),
\]
so for $n \ge k q^{k+1} M$ we have
\[
v_2((n)_k)
\ge
k + \log_2(k q^{k+1}) + \log_2 m.
\]
Therefore,
\[
\frac{\#\{\lambda \vdash m : v_2(\mathfrak{f}_\lambda) < v_2((n)_k)\}}{p(m)}
<
\varepsilon,
\]
by the choice of $M$.

\par Next suppose that $\max_{f \in \operatorname{supp}(F)} d(f) > k$. Let $f_1 \in \operatorname{supp}(F)$ have maximal degree. Then
\[
\begin{aligned}
n &= \sum_{f \in \Phi} d(f) F(f) \\
&> k F(f_1) + \sum_{f \ne f_1} F(f) \\
&= (k-1)F(f_1) + \sum_{f \in \Phi} F(f).
\end{aligned}
\]
Since $F(f_1)\ge 1$, it follows that
\[
n-k \ge \sum_{f \in \Phi} F(f).
\]
Thus $\frac{n!}{\prod_f F(f)!}$ is divisible by $(n)_k$, and hence for any $\boldsymbol{\mu} \mapsto F$ we have
\[
v_2(d_{\boldsymbol{\mu}})
\ge
v_2\left(\frac{n!}{\prod_f F(f)!}\right)
\ge
v_2((n)_k).
\]
Therefore,
\[
\frac{\# \{\boldsymbol{\mu}\in \mathfrak{Y}_n : \boldsymbol{\mu}\mapsto F,\ v_2(d_{\boldsymbol{\mu}})< v_2((n)_k)\}}
{\# \{\boldsymbol{\mu}\in \mathfrak{Y}_n : \boldsymbol{\mu}\mapsto F\}}
= 0.
\]
\noindent Combining the two cases completes the proof.
\end{proof}

\begin{proposition}\label{v(n)_k) propn}
Let $v:=v_2$ and for any $k$, let $(n)_k := n!/(n-k)!$. Then
\[
\lim_{n\to\infty}
\frac{
\#\{\boldsymbol{\mu} \in \mathfrak{Y}_n : v(d_{\boldsymbol{\mu}}) < v((n)_k)\}
}{|\mathfrak{Y}_n|}
= 0.
\]
\end{proposition}

\begin{proof}
Fix $\varepsilon>0$. Let $N=N(\varepsilon)$ be as in Lemma~\ref{mu to F lemma}. For all $n \ge N$, we decompose according to the map $\boldsymbol{\mu} \mapsto F$:
\[
\frac{
\#\{\boldsymbol{\mu} \in \mathfrak{Y}_n : v(d_{\boldsymbol{\mu}}) < v((n)_k)\}
}{|\mathfrak{Y}_n|}
=
\sum_{F \in \mathscr{G}_n(\Phi)}
\frac{
\#\{\boldsymbol{\mu} \in \mathfrak{Y}_n : \boldsymbol{\mu} \mapsto F,\ v(d_{\boldsymbol{\mu}}) < v((n)_k)\}
}{|\mathfrak{Y}_n|}.
\]
For each $F$, we write
\[
\begin{split}&\frac{
\#\{\boldsymbol{\mu} \in \mathfrak{Y}_n : \boldsymbol{\mu} \mapsto F,\ v(d_{\boldsymbol{\mu}}) < v((n)_k)\}
}{|\mathfrak{Y}_n|}\\
= &
\frac{
\#\{\boldsymbol{\mu} \in \mathfrak{Y}_n : \boldsymbol{\mu} \mapsto F,\ v(d_{\boldsymbol{\mu}}) < v((n)_k)\}
}{\#\{\boldsymbol{\mu} \in \mathfrak{Y}_n : \boldsymbol{\mu} \mapsto F\}}
\cdot
\frac{
\#\{\boldsymbol{\mu} \in \mathfrak{Y}_n : \boldsymbol{\mu} \mapsto F\}
}{|\mathfrak{Y}_n|}.
\end{split}
\]
By Lemma~\ref{mu to F lemma}, the first factor is $<\varepsilon$ for all $F \in \mathscr{G}_n(\Phi)$, provided $n \ge N$. Hence
\[
\frac{
\#\{\boldsymbol{\mu} \in \mathfrak{Y}_n : v(d_{\boldsymbol{\mu}}) < v((n)_k)\}
}{|\mathfrak{Y}_n|}
<
\varepsilon \sum_{F \in \mathscr{G}_n(\Phi)}
\frac{
\#\{\boldsymbol{\mu} \in \mathfrak{Y}_n : \boldsymbol{\mu} \mapsto F\}
}{|\mathfrak{Y}_n|}
=
\varepsilon.
\]
Since $\varepsilon>0$ is arbitrary, the result follows.
\end{proof}

\begin{proposition}\label{Propn limit is 0 Y_n}
For every integer $r \ge 1$,
\[
\lim_{n\to\infty}
\frac{
\#\{\boldsymbol{\mu} \in \mathfrak{Y}_n : 2^r \nmid \chi_{\boldsymbol{\mu}}(g)\}
}{|\mathfrak{Y}_n|}
= 0.
\]
\end{proposition}

\begin{proof}
Let $g \in G_{n_0}$ be fixed. From Lemma \ref{valuation lemma}, we have that
\[
\frac{\chi_{\boldsymbol{\mu}}(g)}{d_{\boldsymbol{\mu}}}
\cdot [G_n : G_{n-n_0}]
\]
is an algebraic integer. It follows that if
\[
v_2(d_{\boldsymbol{\mu}})
\;\ge\;
r + v_2\!\left(\prod_{i=0}^{n_0-1} (q^{n-i}-1)\right),
\]
then $2^r \mid \chi_{\boldsymbol{\mu}}(g)$. Set
\[
r_0 := r + n_0\, v(q^2 - 1).
\]
As in the proof of \cite[Proposition 4.7]{shah2024divisibility}, we have that
\[
\{\boldsymbol{\mu} \in \mathfrak{Y}_n : 2^r \nmid \chi_{\boldsymbol{\mu}}(g)\}
\subset
\{\boldsymbol{\mu} \in \mathfrak{Y}_n : v(d_{\boldsymbol{\mu}}) < v((n)_{n_0 + 2r_0})\}.
\]
Applying Proposition \ref{v(n)_k) propn} with $k = n_0 + 2r_0$, we obtain
\[
\lim_{n\to\infty}
\frac{
\#\{\boldsymbol{\mu} \in \mathfrak{Y}_n : 2^r \nmid \chi_{\boldsymbol{\mu}}(g)\}
}{|\mathfrak{Y}_n|}
= 0,
\]
as required.
\end{proof}
We now derive consequences of the formulas for Stiefel--Whitney classes in terms of character values on elements of order $2$, together with the asymptotic divisibility results of the previous section. We show that, in the large rank limit, almost all irreducible orthogonal representations of $\op{GL}_n(\F_q)$ have trivial low-degree Stiefel--Whitney classes.

\begin{theorem}\label{thm asymp SW}
Let $q$ be odd.

\begin{enumerate}
\item One has
\[
\lim_{n\to\infty}
\frac{
\#\{\pi \in \op{OIrr}(G_n) : w_1(\pi)=w_2(\pi)=0\}
}{
|\op{OIrr}(G_n)|
}
= 1.
\]

\item If $q \equiv 1 \pmod{4}$, then
\[
\lim_{n\to\infty}
\frac{
\#\{\pi \in \op{OIrr}(G_n) : w_1(\pi)=w_2(\pi)=w_4(\pi)=0\}
}{
|\op{OIrr}(G_n)|
}
= 1.
\]
\end{enumerate}
\end{theorem}

\begin{proof}
The result follows directly from \eqref{w_1 basic fact}, Theorem \ref{theorem about w_2 and w_4} and Proposition \ref{Propn limit is 0 Y_n}.
\end{proof}

\section{Large $q$ limit for $\GL_2$}\label{s 4}

In this section we study the asymptotic behavior, as $q \to \infty$ with $q$ odd, of the values of irreducible orthogonal characters of $\GL_2(\F_q)$ on a fixed semisimple element. Our goal is to understand the frequency with which powers of $2$ divide such character values. This provides a sharp contrast with the results in the previous section.

\subsection{Classification of irreducible representations}

Let $G_q := \GL_2(\F_q)$. We recall the classification of irreducible complex representations of $G_q$, following \cite{BushnellHenniart}. We fix the standard subgroups:
\[
B = \left\{ \begin{pmatrix} a & b \\ 0 & d \end{pmatrix} \right\}, \quad
N = \left\{ \begin{pmatrix} 1 & x \\ 0 & 1 \end{pmatrix} \right\}, \quad
T = \left\{ \begin{pmatrix} a & 0 \\ 0 & d \end{pmatrix} \right\}, \quad
Z = \left\{ \begin{pmatrix} a & 0 \\ 0 & a \end{pmatrix} \right\}.
\]
Thus $B = TN$ and $N \simeq (\F_q,+)$. Irreducible representations of $G_q$ fall into the following classes:

\begin{enumerate}
\item \textbf{One-dimensional representations.} This is the family of $1$-dimensional representations which factor through the determinant:
\[
\pi = \psi \circ \det,
\]
where $\psi : \F_q^\times \to \mathbb{C}^\times$ is a character. Since $|\F_q^\times| = q-1$, there are exactly $q-1$ such representations.

\item \textbf{Principal series representations.} Let $\chi_1, \chi_2 : \F_q^\times \to \mathbb{C}^\times$ be characters. Define a character of $T$ by
\[
\chi = \chi_1 \otimes \chi_2, \qquad 
\chi\!\left(\begin{pmatrix} a & 0 \\ 0 & d \end{pmatrix}\right) = \chi_1(a)\chi_2(d),
\]
and inflate it to $B$. Consider
\[
\pi(\chi_1,\chi_2) := \Ind_B^{G_q} \chi.
\]
Then one has that
\begin{itemize}
\item $\Ind_B^{G_q} \chi$ is irreducible if and only if $\chi_1 \neq \chi_2$,
\item if $\chi_1 = \chi_2$, then $\Ind_B^{G_q} \chi$ has length $2$.
\end{itemize}

In the reducible case $\chi_1=\chi_2=\psi$, one has
\[
\Ind_B^{G_q}(\psi \otimes \psi)
= (\psi \circ \det) \oplus \left(\St \otimes (\psi \circ \det)\right),
\]
where $\op{St}$ is the \emph{Steinberg representation}.
If $\chi_1 \neq \chi_2$, the representation is irreducible of dimension
\[
\dim \pi(\chi_1,\chi_2) = [G_q : B] = q+1.
\]
\noindent Since $\pi(\chi_1,\chi_2) \cong \pi(\chi_2,\chi_1)$, the number of distinct such irreducible Steinberg representations is
\[
\frac{(q-1)(q-2)}{2}.
\]

\item \textbf{Steinberg representations.} The Steinberg representation $\St$ is defined by $\Ind_B^{G_q}(1) = 1 \oplus \St$. It is irreducible of dimension $q$. More generally, for any character $\psi : \F_q^\times \to \mathbb{C}^\times$, one obtains
\[
\St \otimes (\psi \circ \det),
\]
which are all irreducible of dimension $q$. Thus there are $(q-1)$ such representations.
\item \textbf{Cuspidal representations.} An irreducible representation $\pi$ is called \emph{cuspidal} if it does not contain the trivial character of $N$. Let $\F_{q^2}/\F_q$ be the quadratic extension, and let $\theta : \F_{q^2}^\times \to \mathbb{C}^\times$ be a character. Denote by $\theta^q(x) := \theta(x^q)$. A character $\theta$ is called \emph{regular} if $\theta^q \neq \theta$. Using the embedding $\F_{q^2}^\times \hookrightarrow G_q$ arising from its action on $\F_{q^2}$ as a $2$-dimensional $\F_q$-vector space, one constructs a virtual representation
\[
\pi_\theta = \Ind_{ZN}^{G_q}(\theta \otimes \psi) - \Ind_{\F_{q^2}^\times}^{G_q} \theta,
\]
where $\psi$ is a non-trivial character of $N$. Then, the following assertions hold:
\begin{itemize}
\item $\pi_\theta$ is irreducible of dimension $q-1$,
\item $\pi_\theta \cong \pi_{\theta'}$ if and only if $\theta' = \theta$ or $\theta' = \theta^q$,
\item every cuspidal representation arises this way.
\end{itemize}
\noindent Since there are $q^2-1$ characters of $\F_{q^2}^\times$, and exactly $(q-1)$ of them satisfy $\theta^{q-1}=1$, the number of regular characters is \[q^2-1 - (q-1) = q(q-1).\]
Dividing by the equivalence $\theta \sim \theta^q$, we obtain $\frac{q(q-1)}{2}$ cuspidal representations.

\end{enumerate}

In summary, the total number of irreducible representations is
\[
(q-1) + \frac{(q-1)(q-2)}{2} + (q-1) + \frac{q(q-1)}{2}
= q^2 - 1,
\]
in agreement with the number of conjugacy classes of $G_q$.

Recall that for any finite-dimensional representation $\pi$ of $G_q$, the contragredient representation $\pi^\vee$ has character $\chi_{\pi^\vee}(g)=\overline{\chi_\pi(g)}$. We say that $\pi$ is \emph{self-dual} if $\pi \simeq \pi^\vee$.

\begin{lemma}\label{ind dual iso}
Let $G$ be a finite group, $H \subset G$ a subgroup, and let $(\sigma,V)$ be a finite-dimensional complex representation of $H$. Then there is a natural isomorphism of $G$-representations
\[
(\Ind_H^G \sigma)^\vee \;\simeq\; \Ind_H^G(\sigma^\vee).
\]
\end{lemma}

\begin{proof}
Recall that
\[
\Ind_H^G V
=
\{ f: G \to V \mid f(hg)=\sigma(h)f(g)\ \text{for all } h\in H,\, g\in G \},
\]
has $G$-action given by
\[
(g\cdot f)(x)=f(xg), \quad\text{for}\quad g,x\in G.
\]
Similarly,
\[
\Ind_H^G V^\vee
=
\{ \lambda: G \to V^\vee \mid \lambda(hg)=\sigma^\vee(h)\lambda(g)\ \text{for all } h\in H,\, g\in G \},
\]
with the same right translation action.

Define a pairing
\[
\langle\cdot,\cdot\rangle : \Ind_H^G V \times \Ind_H^G V^\vee \longrightarrow \mathbb{C}
\]
by
\[
\langle f,\lambda\rangle
:=
\sum_{x \in H\backslash G} \lambda(x)\bigl(f(x)\bigr).
\]
For $g\in G$,
\[
\langle g\cdot f,\lambda\rangle
=
\sum_{x\in H\backslash G} \lambda(x)\bigl(f(xg)\bigr).
\]
Making the change of variables $y=xg$, which permutes the cosets $H\backslash G$, we obtain
\[
\langle g\cdot f,\lambda\rangle
=
\sum_{y\in H\backslash G} \lambda(yg^{-1})\bigl(f(y)\bigr)
=
\langle f, g^{-1}\cdot \lambda\rangle,
\]
since $(g^{-1}\cdot \lambda)(y)=\lambda(yg^{-1})$. Thus the pairing is $G$-invariant.
\par This pairing is clearly bilinear, and nondegenerate. If $f\neq 0$, then there exists $x$ such that $f(x)\neq 0$, and one may choose $\lambda$ supported on the coset $Hx$ so that $\lambda(x)(f(x))\neq 0$. Thus, $\langle f, \lambda\rangle \neq 0$. Similarly, if $\lambda\neq 0$, one may choose $f$ such that $\langle f, \lambda\rangle \neq 0$. \par Therefore the pairing identifies $\Ind_H^G V^\vee$ with $(\Ind_H^G V)^\vee$ as $G$-representations, yielding the desired isomorphism
\[
(\Ind_H^G \sigma)^\vee \simeq \Ind_H^G(\sigma^\vee).
\]
\end{proof}

\begin{proposition}\label{steinberg self sual prop}
Let $G=\GL_2(\F_q)$ and let $\St$ denote the Steinberg representation. Then
\[
\St^\vee \simeq \St.
\]
\end{proposition}

\begin{proof}
Let $B\subset G$ be the Borel subgroup of upper triangular matrices, and consider the (unnormalized) induced representation
\[
\pi := \Ind_B^G \mathbf{1}.
\]
It is well known that $\pi$ contains the trivial representation $\mathbf{1}$ as a subrepresentation (via constant functions), and that the Steinberg representation is defined as the quotient
\[
\St := \pi / \mathbf{1}.
\]
\noindent We first observe that $\pi$ is self-dual. Indeed, by Lemma~\ref{ind dual iso},
\[
\pi^\vee = (\Ind_B^G \mathbf{1})^\vee \simeq \Ind_B^G(\mathbf{1}) = \pi.
\]
\noindent Realize $\pi$ as
\[
\pi = \{ f: G \to \mathbb{C} \mid f(bg)=f(g)\ \text{for all } b\in B \}.
\]
Define
\[
\langle f_1, f_2 \rangle := \sum_{x \in B\backslash G} f_1(x) f_2(x).
\]
As in the proof of Lemma \ref{ind dual iso}, $\pi \simeq \pi^\vee$ via this pairing.

\par Consider the subrepresentation $\mathbf{1} \subset \pi$ consisting of constant functions. We have that
\[
\mathbf{1}^\perp
=
\left\{ f \in \pi \;\middle|\; \sum_{x\in B\backslash G} f(x)=0 \right\}.
\]
This subspace has codimension one, hence
\[
\mathbf{1}^\perp \simeq \St.
\]
\noindent Since the pairing on $\pi$ is nondegenerate and $G$-invariant, it induces a nondegenerate $G$-invariant pairing on $\mathbf{1}^\perp$. Therefore $\mathbf{1}^\perp$ is self-dual, and thus, $\St^\vee \simeq \St$, as claimed.
\end{proof}

\begin{lemma}\label{duality for cuspidal}
Let $G_q=\GL_2(\F_q)$ and let $\theta:\F_{q^2}^\times\to\mathbb{C}^\times$ be a character with $\theta\neq\theta^q$. Let $\pi_\theta$ be the associated cuspidal representation, realized as the virtual representation
\[
\pi_\theta = \Ind_{ZN}^{G_q}(\theta_\psi) - \Ind_E^{G_q}(\theta).
\]
Then
\[
\pi_\theta^\vee \simeq \pi_{\theta^{-1}}.
\]
\end{lemma}

\begin{proof}
We dualize the defining expression term by term. For any subgroup $H\subset G_q$ and any finite-dimensional representation $\sigma$ of $H$, we have
\[
(\Ind_H^{G_q} \sigma)^\vee \simeq \Ind_H^{G_q}(\sigma^\vee).
\]
Applying this to each term gives
\[
\pi_\theta^\vee
=
\Ind_{ZN}^{G_q}(\theta_\psi^\vee)
-
\Ind_E^{G_q}(\theta^\vee).
\]
\noindent Since $\theta$ is one-dimensional,
\[
\theta^\vee = \theta^{-1}.
\]
Thus
\[
(\Ind_E^{G_q} \theta)^\vee \simeq \Ind_E^{G_q}(\theta^{-1}).
\]
\noindent Recall that $\theta_\psi$ is defined on $ZN$ by
\[
\theta_\psi(zn) = \theta(z)\,\psi(n),
\]
where $\psi:N\to\mathbb{C}^\times$ is a nontrivial additive character (via the standard identification $N\simeq \F_q$). Thus
\[
\theta_\psi^\vee(zn)
=
\theta(z)^{-1}\,\psi(n)^{-1}.
\]
Since $\psi^{-1}$ is again a nontrivial additive character of $N$, it differs from $\psi$ by conjugation in $G_q$: more precisely, there exists $g\in G_q$ such that
\[
\psi(n)^{-1} = \psi(gng^{-1}) \quad \text{for all } n\in N.
\]
It follows that $\theta_\psi^\vee$ is conjugate (as a representation of $ZN$) to $(\theta^{-1})_\psi$. Therefore
\[
\Ind_{ZN}^{G_q}(\theta_\psi^\vee)
\simeq
\Ind_{ZN}^{G_q}((\theta^{-1})_\psi),
\]
since induction is invariant under conjugation of the inducing representation. Combining the above, we obtain
\[
\pi_\theta^\vee
=
\Ind_{ZN}^{G_q}((\theta^{-1})_\psi)
-
\Ind_E^{G_q}(\theta^{-1})
=
\pi_{\theta^{-1}}.
\]
\noindent This proves the desired isomorphism.
\end{proof}

\begin{proposition}\label{prop self dual GL2}
Let $\pi$ be an irreducible complex representation of $G_q=\GL_2(\F_q)$. Then:

\begin{enumerate}
\item If $\pi = \psi \circ \det$ is one-dimensional, then
\[
\pi^\vee \simeq \psi^{-1} \circ \det.
\]
In particular, $\pi$ is self-dual if and only if $\psi^2=1$.

\item If $\pi = \op{Ind}_B^{G_q}(\chi_1 \otimes \chi_2)$ is a principal series representation with $\chi_1 \neq \chi_2$, then
\[
\pi^\vee \simeq \op{Ind}_B^{G_q}(\chi_1^{-1} \otimes \chi_2^{-1}).
\]
Thus $\pi$ is self-dual if and only if
\[
\{\chi_1,\chi_2\} = \{\chi_1^{-1},\chi_2^{-1}\}.
\]

\item If $\pi = \St \otimes (\psi \circ \det)$ is a twist of the Steinberg representation, then
\[
\pi^\vee \simeq \St \otimes (\psi^{-1} \circ \det).
\]
In particular, $\pi$ is self-dual if and only if $\psi^2=1$.

\item If $\pi=\pi_\theta$ is a cuspidal representation associated to a character $\theta:\F_{q^2}^\times \to \mathbb{C}^\times$ with $\theta \neq \theta^q$, then
\[
\pi_\theta^\vee \simeq \pi_{\theta^{-1}}.
\]
In particular, $\pi_\theta$ is self-dual if and only if
\[
\theta^q = \theta^{-1}.
\]
\end{enumerate}
\end{proposition}

\begin{proof}
\par Part (1) is immediate. 
\par For part (2), let $\pi = \op{Ind}_B^{G_q}(\chi_1 \otimes \chi_2)$. By Lemma \ref{ind dual iso}, one has
\[
\left(\op{Ind}_B^{G_q}(\chi_1 \otimes \chi_2)\right)^\vee
\simeq
\op{Ind}_B^{G_q}(\chi_1^{-1} \otimes \chi_2^{-1})
\]
The Weyl group symmetry gives
\[
\op{Ind}_B^{G_q}(\chi_1 \otimes \chi_2)
\simeq
\op{Ind}_B^{G_q}(\chi_2 \otimes \chi_1).
\]
It follows that $\pi \simeq \pi^\vee$ if and only if the unordered pair $\{\chi_1,\chi_2\}$ is invariant under inversion, i.e.\ $\{\chi_1,\chi_2\}=\{\chi_1^{-1},\chi_2^{-1}\}$.
\par For part (3), Proposition \ref{ind dual iso} asserts that $\St^\vee \simeq \St$. Hence,
\[
(\St \otimes (\psi \circ \det))^\vee
\simeq
\St \otimes (\psi^{-1} \circ \det).
\]
\par Part (4) follows from Lemma \ref{duality for cuspidal}.
\end{proof}
\begin{lemma}\label{det pi formula}
Let $G$ be a finite group, $H \subset G$ a subgroup, and let $\sigma:H \to \mathbb{C}^\times$ be a one-dimensional representation. Let $\pi = \Ind_H^G \sigma$. Fix a set of representatives $\{x_i\}$ for $H\backslash G$. For $g\in G$, write
\[
x_i g = h_i(g)\, x_{\rho(i)}
\]
with $h_i(g)\in H$ and $\rho$ a permutation of the index set. Then
\[
\det(\pi)(g)
=
\op{sgn}(\rho)\prod_i \sigma\bigl(h_i(g)\bigr).
\]
\end{lemma}

\begin{proof}
Let $\{e_i\}$ be the standard basis of $\Ind_H^G \sigma$ corresponding to the cosets $Hx_i$. By definition of induction,
\[
g\cdot e_i = \sigma\bigl(h_i(g)\bigr)\, e_{\rho(i)}.
\]
Thus the matrix of $\pi(g)$ is a permutation matrix corresponding to $\rho$, with nonzero entries $\sigma(h_i(g))$ in positions $(i,\rho(i))$. The determinant of such a matrix is the sign of the permutation multiplied by the product of the nonzero entries, hence
\[
\det(\pi)(g)
=
\op{sgn}(\rho)\prod_i \sigma\bigl(h_i(g)\bigr),
\]
as claimed.
\end{proof}
Define the quadratic character $\mu:\F_q^\times \longrightarrow \{\pm1\}$ by: \[\mu(a):=
\begin{cases}
1 & \text{if } a \in (\F_q^\times)^2,\\
-1 & \text{otherwise}.
\end{cases}
\]
Extend $\mu$ to a character of $G_q=\GL_2(\F_q)$ as follows
\[
\mu(g):=\mu(\det g).
\]
\begin{lemma}\label{mu lemma}
For $g\in G_q=\GL_2(\F_q)$, let $\rho_{\mathbb{P}^1}(g)$ be the permutation of $\mathbb{P}^1(\F_q)$ induced by $g$. Then
\[
\op{sgn}(\rho_{\mathbb{P}^1}(g))=\mu(\det g).
\]
\end{lemma}

\begin{proof}
The action of $G_q$ on $\mathbb{P}^1(\F_q)$ factors through $\op{PGL}_2(\F_q)$, since scalar matrices act trivially. Thus the map
\[
g \longmapsto \op{sgn}(\rho_{\mathbb{P}^1}(g))
\]
factors through $\op{PGL}_2(\F_q)$ and defines a group homomorphism
\[
\op{PGL}_2(\F_q)\longrightarrow \{\pm1\}.
\]

For $q$ odd, $\op{PSL}_2(\F_q)$ is a normal subgroup of index $2$ in $\op{PGL}_2(\F_q)$, and one has a canonical identification
\[
\op{PGL}_2(\F_q)/\op{PSL}_2(\F_q)\;\simeq\; \F_q^\times/(\F_q^\times)^2,
\]
via the determinant map. It follows that any homomorphism $\op{PGL}_2(\F_q)\to\{\pm1\}$ is determined by a quadratic character of $\F_q^\times$, and hence there exists a character $\chi:\F_q^\times\to\{\pm1\}$ such that
\[
\op{sgn}(\rho_{\mathbb{P}^1}(g))=\chi(\det g).
\]

It remains to identify $\chi$. For this, consider the element
\[
g=\begin{pmatrix}a & 0 \\ 0 & 1\end{pmatrix}.
\]
Then $g$ acts on $\mathbb{P}^1(\F_q)=\F_q\cup\{\infty\}$ by
\[
x\mapsto ax, \qquad \infty\mapsto\infty.
\]
Thus $\rho_{\mathbb{P}^1}(g)$ fixes $\infty$ and permutes $\F_q$ via multiplication by $a$. Since $0$ is also fixed, the sign of $\rho_{\mathbb{P}^1}(g)$ is equal to the sign of the permutation of $\F_q^\times$ given by $x\mapsto ax$.

Choose a generator $\gamma$ of the cyclic group $\F_q^\times$, so that $\F_q^\times=\{\gamma^k : 0\le k\le q-2\}$. Writing $a=\gamma^m$, the map $x\mapsto ax$ corresponds to the permutation $k\mapsto k+m$ of $\mathbb{Z}/(q-1)\mathbb{Z}$. This permutation is a cycle of length $(q-1)/\gcd(m,q-1)$ repeated $\gcd(m,q-1)$ times, and its sign is
\[
(-1)^{(q-1)-\gcd(m,q-1)}=(-1)^{\gcd(m,q-1)}=(-1)^m=\mu(a).
\]
Thus conclude that
\[
\op{sgn}(\rho_{\mathbb{P}^1}(g))=\mu(a)=\mu(\det g)
\]
for all such diagonal elements, and hence for all $g\in G_q$.
\end{proof}
\begin{proposition}
Let $G_q=\GL_2(\F_q)$ and let $\pi$ be a self-dual irreducible one-dimensional or principal series complex representation. Then the first Stiefel--Whitney class $w_1(\pi)$, identified with $\det(\pi):G_q\to\{\pm1\}$, is given as follows:
\[
\begin{cases}
\chi & \text{if } \pi=\chi\circ\det,\\
\mu\circ \det & \text{if } \pi=\Ind_B^{G_q}(\chi_1\otimes\chi_2),\\
 \mu\circ \det & \text{if } \pi=\St\otimes(\psi\circ\det).\\
\end{cases}
\]
\end{proposition}

\begin{proof}
We compute $\det(\pi)$ in each case.

\medskip

\noindent
\textit{(i) One-dimensional representations.}
If $\pi=\chi\circ\det$, then $\det(\pi)=\chi\circ\det$ by definition.

\medskip

\noindent
\textit{(ii) Principal series.}
Let $\pi=\Ind_B^{G_q}(\chi_1\otimes\chi_2)$, and write $\sigma=\chi_1\otimes\chi_2$. By the determinant formula for induced representations, fixing coset representatives $\{x_i\}$ for $B\backslash G_q$, we have
\[
\det(\pi)(g)
=
\op{sgn}(\rho_g)\prod_i \sigma(h_i(g)),
\quad\text{where } x_i g = h_i(g)\,x_{\rho_g(i)}.
\]
Since $\sigma$ depends only on diagonal entries, we have
\[
\sigma(h_i(g)) = \chi_1(a_i(g))\chi_2(d_i(g)),
\]
where $a_i(g)d_i(g)=\det(g)$. Hence
\[
\prod_i \sigma(h_i(g))
=
\prod_i (\chi_1\chi_2)(\det g)
=
(\chi_1\chi_2)(\det g)^{q+1}.
\]
Since the representation is self dual, $\chi_1\chi_2=1$. Hence by Lemma \ref{mu lemma}, we find that 
\[\det(\pi)=\mu\circ \det.\]
\medskip

\noindent
\textit{(iii) Steinberg twists.}
Using the virtual identity
\[
\St = \Ind_B^{G_q}\mathbf{1} - \mathbf{1},
\]
we compute determinants multiplicatively:
\[
\det(\St)
=
\frac{\det(\Ind_B^{G_q}\mathbf{1})}{\det(\mathbf{1})}=\mu\circ \det.
\]
Now if $\pi=\St\otimes(\psi\circ\det)$, then
\[
\det(\pi)
=
\det(\St)\cdot (\psi\circ\det)^{\dim \St}
=
(\mu\circ \det) (\psi\circ\det)^q=(\mu\circ \det) (\psi\circ\det)=\mu\psi\circ \det.
\]
\end{proof}

\par Recall that $h_1=\diag(-1,1)$ and $h_2=\diag(-1,-1)=-I$. 
\begin{proposition}\label{prop self dual GL2 with h2}
Let $\pi$ be an irreducible complex representation of $G_q=\GL_2(\F_q)$. Then:

\begin{enumerate}
\item If $\pi = \psi \circ \det$ is one-dimensional and $\psi^2=1$. Then we have that
\[
\chi_\pi(h_1)=\psi(-1) \quad\text{ and }\quad \chi_\pi(h_2)=\psi(1)=1.
\]

\item If $\pi = \op{Ind}_B^{G_q}(\chi_1 \otimes \chi_2)$ is a self dual principal series representation with $\chi_1 \neq \chi_2$. Then, \[
\chi_\pi(h_1)=\chi_1(-1)+\chi_2(-1)
\quad\text{ and }\quad
\chi_\pi(h_2)=(q+1)\chi_1(-1)\chi_2(-1).
\]

\item If $\pi = \St \otimes (\psi \circ \det)$ is a twist of the Steinberg representation for which $\psi^2=1$, we have that
\[
\chi_\pi(h_1)=\psi(-1) \quad\text{and}\quad \chi_\pi(h_2)=q\cdot \psi(1)=q.
\]

\item If $\pi=\pi_\theta$ is a self-dual cuspidal representation associated to a character $\theta:\F_{q^2}^\times \to \mathbb{C}^\times$ with $\theta \neq \theta^q$. Then,
\[
\chi_{\pi_\theta}(h_1)=0
\quad\text{and}\quad
\chi_{\pi_\theta}(h_2)=(q-1)\,\theta(-1).
\]
\end{enumerate}
\end{proposition}

\begin{proof}
We treat each case separately.

\par

\noindent
(1) Since $\det(h_1)=-1$ and $\det(h_2)=1$, we obtain
\[
\chi_\pi(h_1)=\psi(-1) \quad \text{and}\quad\chi_\pi(h_2)=\psi(1)=1.
\]

\par

\noindent
(2) We use the standard formula for the character of an induced representation:
\[
\chi_{\Ind_B^{G_q}\sigma}(g)
=
\frac{1}{|B|}
\sum_{\substack{x\in G_q \\ x^{-1}gx \in B}}
\chi_\sigma(x^{-1}gx),
\]
where $\sigma=\chi_1\otimes\chi_2$. The element $h_1$ is semisimple with distinct eigenvalues. There are exactly two such $B$-conjugacy representatives, namely
\[
\begin{pmatrix}-1 & 0 \\ 0 & 1\end{pmatrix}
\quad \text{and} \quad
\begin{pmatrix}1 & 0 \\ 0 & -1\end{pmatrix}.
\]Thus
\[
\chi_\pi(h_1)
=
\chi_\sigma\!\left(\begin{pmatrix}-1 & 0 \\ 0 & 1\end{pmatrix}\right)
+
\chi_\sigma\!\left(\begin{pmatrix}1 & 0 \\ 0 & -1\end{pmatrix}\right),
\]
and we obtain
\[
\chi_\pi(h_1)
=
\chi_1(-1)\chi_2(1) + \chi_1(1)\chi_2(-1)
=
\chi_1(-1)+\chi_2(-1).
\]
For $h_2=-I$, which is central in $G_q$, we have $x^{-1}h_2x=h_2$ for all $x\in G_q$, and hence
\[
\chi_\pi(h_2)
=
\frac{1}{|B|}
\sum_{x\in G_q} \chi_\sigma(h_2)
=
\frac{|G_q|}{|B|}\,\chi_\sigma(h_2).
\]
Since $\dim(\pi)=[G_q:B]=q+1$, it follows that
\[
\chi_\pi(h_2)
=
(q+1)\chi_\sigma(-I)
=
(q+1)\chi_1(-1)\chi_2(-1).
\]

\par

\noindent
(3) The computation for Steinberg representations follows along the same lines as part (2).
\par

\noindent
(4) We use the virtual representation
\[
\pi_\theta = \Ind_{ZN}^{G_q}(\theta_\psi) - \Ind_E^{G_q}(\theta),
\]
and compute characters using the standard formula for induced representations:
\[
\chi_{\Ind_H^{G_q}\sigma}(g)
=
\frac{1}{|H|}
\sum_{\substack{x\in G_q \\ x^{-1}gx \in H}}
\chi_\sigma(x^{-1}gx).
\]
The element $h_1$ is split semisimple with distinct eigenvalues. We claim that neither $ZN$ nor $E$ contains any conjugate of $h_1$. First, every element of $ZN$ is unipotent up to scalar, hence has a single eigenvalue; thus no conjugate of $h_1$ lies in $ZN$. Therefore
\[
\chi_{\Ind_{ZN}^{G_q}(\theta_\psi)}(h_1)=0.
\]
Diagonalizable elements in $E\simeq \F_{q^2}^\times$ are the elements in $\F_q^\times$. Thus the only diagonalizable elements in $E$ are scalar. Since $h_1$ is non-scalar and diagonalizable over $\F_q$, it is not conjugate to any element of $E$. Hence, $\chi_{\Ind_E^{G_q}(\theta)}(h_1)=0$ and we deduce that
\[
\chi_{\pi_\theta}(h_1)=0.
\]
\noindent Since $h_2=-I$ is central, the induced character formula simplifies to
\[
\chi_{\Ind_H^{G_q}\sigma}(h_2)
=
\frac{|G_q|}{|H|}\,\chi_\sigma(h_2)
\]
provided $h_2\in H$, and is $0$ otherwise. First, $h_2\in Z\subset ZN$, so
\[
\chi_{\Ind_{ZN}^{G_q}(\theta_\psi)}(h_2)
=
\frac{|G_q|}{|ZN|}\,\theta_\psi(h_2).
\]
Since $\theta_\psi(zn)=\theta(z)\psi(n)$ and $h_2=-I\in Z$, we have
\[
\theta_\psi(h_2)=\theta(-1).
\]
A direct computation shows
\[
\frac{|G_q|}{|ZN|} = q^2-1,
\]
so
\[
\chi_{\Ind_{ZN}^{G_q}(\theta_\psi)}(h_2)=(q^2-1)\theta(-1).
\]Next, $h_2\in E$, so similarly
\[
\chi_{\Ind_E^{G_q}(\theta)}(h_2)
=
\frac{|G_q|}{|E|}\,\theta(-1).
\]
Since $|E|=q^2-1$, we have
\[
\frac{|G_q|}{|E|} = q(q-1),
\]
and hence
\[
\chi_{\Ind_E^{G_q}(\theta)}(h_2)=q(q-1)\theta(-1).
\]Therefore
\[
\chi_{\pi_\theta}(h_2)
=
(q^2-1)\theta(-1) - q(q-1)\theta(-1)
=
(q-1)\theta(-1).
\]
\end{proof}

\begin{remark}
For any irreducible character $\chi$ of $G_q$, one has
\[
|\chi(h_1)| \le 2.
\]
In particular, if $r \ge 2$, then
\[
2^r \mid \chi(h_1) \quad \Longleftrightarrow \quad \chi(h_1)=0,
\]
since the only integer divisible by $2^r$ with absolute value at most $2$ is $0$.
\end{remark}

\begin{proposition}\label{prop SW classes GL2 q=1mod4 clean}
Let $q\equiv 1 \pmod{4}$, let $G_q=\GL_2(\F_q)$, and let $\pi$ be a self-dual irreducible complex representation of $G_q$. Then the following hold.

\begin{enumerate}

\item If $\pi$ is one-dimensional, then
\[
w_2(\pi)=w_4(\pi)=0.
\]

\item Suppose $\dim \pi>1$.

\begin{enumerate}

\item[(a)] If $\pi=\Ind_B^{G_q}(\chi_1\otimes \chi_1^{-1})$, then
\[
w_2(\pi)=0
\iff
\begin{cases}
q\equiv 1 \pmod{8}, & \text{if } \chi_1(-1)=1,\\[6pt]
q\equiv 5 \pmod{8}, & \text{if } \chi_1(-1)=-1.
\end{cases}
\]

\item[(b)]
If $\pi=\St\otimes(\psi\circ\det)$ with $\psi^2=1$, then
\[
w_2(\pi)=0 \iff q\equiv 1 \pmod{8}.
\]

\item[(c)]  
If $\pi=\pi_\theta$ is cuspidal and self-dual, then
\[
w_2(\pi)=0 \iff q\equiv 1 \pmod{8}.
\]

\end{enumerate}

\item For $\pi$ with $\dim \pi>1$, in the principal series case $\pi=\Ind_B^{G_q}(\chi_1\otimes\chi_1^{-1})$,
\[
w_4(\pi)=0
\iff
\begin{cases}
q \equiv 1 \pmod{16}, & \text{if } \chi_1(-1)=1,\\[6pt]
q \equiv 13 \pmod{16}, & \text{if } \chi_1(-1)=-1,
\end{cases}
\]
while for Steinberg twists and cuspidal representations one always has
\[
w_4(\pi)=0 \iff q \equiv 1 \pmod{16}.
\]

\end{enumerate}
\end{proposition}
\begin{proof}
Let $h_1=\diag(-1,1)$ and $h_2=-I$. For any representation $\pi$, recall that
\[
m_\pi = \frac{\dim \pi - \chi_\pi(h_1)}{2}.
\]
By Theorem~\ref{theorem about w_2 and w_4}, we have
\[
w_2(\pi)=0 \iff m_\pi \equiv 0 \pmod{4},
\]
and
\[
w_4(\pi)|_D
=
\binom{m_\pi/2}{2}\sum_i t_i^2
+
\frac{\dim \pi - \chi_\pi(h_2)}{8}
\sum_{i<j} t_i t_j.
\]
Thus $w_4(\pi)=0$ if and only if both coefficients vanish modulo $2$.

\medskip

\noindent
\textbf{Case 1: $\pi$ one-dimensional.}  
If $\pi=\psi\circ\det$ with $\psi^2=1$, then $\psi(-1)=1$ since $-1$ is a square in $\F_q^\times$ (as $q\equiv 1 \pmod{4}$). Hence $\chi_\pi(h_1)=1$, $m_\pi=0$, so $w_2(\pi)=0$. Also $\chi_\pi(h_2)=1=\dim\pi$, so $w_4(\pi)=0$.

\medskip

\noindent
\textbf{Case 2: $\dim \pi>1$.}

\smallskip

\noindent
\textit{(a) Principal series.}  
Let $\pi=\Ind_B^{G_q}(\chi_1\otimes\chi_1^{-1})$. Then
\[
\chi_\pi(h_1)=\chi_1(-1)+\chi_1(-1)^{-1}.
\]

If $\chi_1(-1)=1$, then $\chi_\pi(h_1)=2$, and hence
\[
m_\pi=\frac{(q+1)-2}{2}=\frac{q-1}{2}.
\]

If $\chi_1(-1)=-1$, then $\chi_\pi(h_1)=-2$, and hence
\[
m_\pi=\frac{(q+1)-(-2)}{2}=\frac{q+3}{2}.
\]

Therefore
\[
w_2(\pi)=0
\iff
\begin{cases}
\frac{q-1}{2}\equiv 0 \pmod{4}, & \text{if } \chi_1(-1)=1,\\[6pt]
\frac{q+3}{2}\equiv 0 \pmod{4}, & \text{if } \chi_1(-1)=-1,
\end{cases}
\]
which is equivalent to
\[
w_2(\pi)=0
\iff
\begin{cases}
q\equiv 1 \pmod{8}, & \text{if } \chi_1(-1)=1,\\[6pt]
q\equiv 5 \pmod{8}, & \text{if } \chi_1(-1)=-1.
\end{cases}
\]

\smallskip

\noindent
\textit{(b) Steinberg twists.}  
Let $\pi=\St\otimes(\psi\circ\det)$ with $\psi^2=1$. Then $\psi(-1)=1$, and one computes
\[
\chi_\pi(h_1)=0,\qquad \dim \pi=q,
\]
so
\[
m_\pi=\frac{q-1}{2}.
\]
Thus
\[
w_2(\pi)=0 \iff q\equiv 1 \pmod{8}.
\]

\smallskip

\noindent
\textit{(c) Cuspidal representations.}  
Let $\pi=\pi_\theta$ be cuspidal and self-dual. Then $\chi_\pi(h_1)=0$ and $\dim \pi=q-1$, so
\[
m_\pi=\frac{q-1}{2},
\]
and hence
\[
w_2(\pi)=0 \iff q\equiv 1 \pmod{8}.
\]
\par We now analyze the two coefficients in the expression for $w_4(\pi)$. By Theorem~\ref{theorem about w_2 and w_4}, the class $w_4(\pi)$ vanishes if and only if both
\[
\binom{m_\pi/2}{2}
\quad \text{and} \quad
\frac{\dim \pi - \chi_\pi(h_2)}{8}
\]
are zero modulo $2$.

We first consider the binomial coefficient. A direct congruence calculation shows that
\[
\binom{m_\pi/2}{2}\equiv 0 \pmod{2}
\iff m_\pi \equiv 0 , 2\pmod{8}.
\]
From the case-by-case analysis above, one has $m_\pi \in \{(q-1)/2,\,(q+3)/2\}$. Checking these two possibilities separately and noting that $q\equiv 1\pmod{4}$, we find
\[
m_\pi \equiv 0 \pmod{8}
\iff
\begin{cases}
q \equiv 1 \pmod{16}, & \text{if } m_\pi=\frac{q-1}{2},\\[6pt]
q \equiv 13 \pmod{16}, & \text{if } m_\pi=\frac{q+3}{2}.
\end{cases}
\]

We now turn to the second coefficient. For principal series and Steinberg representations, one has \[\chi_\pi(h_2)=\dim \pi,\] so this term vanishes. For cuspidal representations $\pi=\pi_\theta$, one has \[\chi_\pi(h_2)=(q-1)\theta(-1).\] Since $\pi$ is self-dual, the parameter $\theta$ satisfies $\theta^q=\theta^{-1}$, and hence $\theta^{q+1}=1$. Restricting to $\F_q^\times$, it follows that $\theta(a)^2=1$ for all $a\in \F_q^\times$, so $\theta|_{\F_q^\times}$ is quadratic. As $q\equiv 1 \pmod{4}$, the element $-1$ is a square in $\F_q^\times$, and therefore $\theta(-1)=1$. It follows that \[\chi_\pi(h_2)=q-1=\dim \pi,\] and hence the second coefficient also vanishes.

Combining these two observations, we conclude that $w_4(\pi)=0$ if and only if $m_\pi \equiv 0 \pmod{8}$, which yields the stated congruence conditions on $q$.

\end{proof}

\begin{proposition}\label{prop w2 GL2 q=3mod4}
Let $q\equiv 3 \pmod{4}$, let $G_q=\GL_2(\F_q)$, and let $\pi$ be a self-dual irreducible complex representation of $G_q$.

\begin{enumerate}
\item If $\pi$ is one-dimensional, then $w_2(\pi)= 0$.

\item Suppose that $\dim \pi>1$. Then
\[
w_2(\pi)=0
\]
if and only if either:
\begin{enumerate}
    \item[(i)] $\pi$ is a principal series representation and if\begin{itemize}
    \item[(a)] $q\equiv 3\pmod{8}$ and $\chi_1(-1)=1$,
    \item[(b)] $q\equiv 7\pmod{8}$ and $\chi_1(-1)=-1$.
\end{itemize}
\item[(ii)]$\pi=\St\otimes(\psi\circ\det)$ is a Steinberg twist with $\psi^2=1$, $\psi\neq 1$, and $q\equiv 7 \pmod{8}$.
\end{enumerate}
In all other cases, one has $w_2(\pi)\neq 0$.
\end{enumerate}
\end{proposition}
\begin{proof}
Let $h_1=\diag(-1,1)$. Recall that
\[
m_\pi=\frac{\dim \pi - \chi_\pi(h_1)}{2},
\quad
w_2(\pi)=0 \iff m_\pi \equiv 0,1 \pmod{4}.
\]

Since $q\equiv 3 \pmod{4}$, the element $-1$ is not a square in $\F_q^\times$. Hence for any quadratic character $\chi$ of $\F_q^\times$, one has $\chi(-1)=-1$. We compute $m_\pi$ using Proposition~\ref{prop self dual GL2 with h2}.
\medskip

\noindent
\textit{(i) One-dimensional representations.}
If $\pi=\psi\circ\det$ with $\psi\neq 1$ and $\psi^2=1$, then $\psi(-1)=-1$, so
\[
m_\pi=\frac{1-(-1)}{2}=1.
\]
Thus $w_2(\pi)= 0$.

\medskip

\noindent
\textit{(ii) Principal series.}
We have
\[
m_\pi=
\begin{cases}
\frac{q-1}{2} & \text{if } \chi_1(-1)=1,\\[6pt]
\frac{q+3}{2} & \text{if } \chi_1(-1)=-1.
\end{cases}
\]
If $q\equiv 3 \pmod{8}$, then
\[
\frac{q-1}{2}\equiv 1 \pmod{4}
\quad\text{and}\quad
\frac{q+3}{2}\equiv 3 \pmod{4}.
\]
Therefore if $q\equiv 3\pmod{8}$, we have that $w_2(\pi)=0$ if and only if $\chi_1(-1)=1$.
\par On the other hand, if $q\equiv 7 \pmod{8}$, then
\[
\frac{q-1}{2}\equiv 3 \pmod{4},
\qquad
\frac{q+3}{2}\equiv 1 \pmod{4},
\]
so again $w_2(\pi)=0$. if and only if $\chi_1(-1)=-1$. Thus for principal series, one has $w_2(\pi)=0$ if and only if:
\begin{itemize}
    \item[(a)] $q\equiv 3\pmod{8}$ and $\chi_1(-1)=1$,
    \item[(b)] $q\equiv 7\pmod{8}$ and $\chi_1(-1)=-1$.
\end{itemize}

\medskip

\noindent
\textit{(iii) Steinberg twists.}
Here
\[
m_\pi=\frac{q-\psi(-1)}{2}=\frac{q+1}{2}.
\]
If $\psi(-1)=1$, then $m_\pi=\frac{q-1}{2}$; if $\psi(-1)=-1$, then $m_\pi=\frac{q+1}{2}$. If $q\equiv 3 \pmod{8}$, then
\[
\frac{q-1}{2}\equiv 1,\quad \frac{q+1}{2}\equiv 2 \pmod{4},
\]
and hence $w_2(\pi)\neq 0$ in this case.

If $q\equiv 7 \pmod{8}$, then
\[
\frac{q-1}{2}\equiv 3,\quad \frac{q+1}{2}\equiv 0 \pmod{4},
\]
so again $w_2(\pi)=0$ occurs if and only if $\psi\neq 1$.

\medskip

\noindent
\textit{(iv) Cuspidal representations.}
Here
\[
m_\pi=\frac{q-1}{2}.
\]
If $q\equiv 3 \pmod{8}$, then $m_\pi\equiv 1 \pmod{4}$, while if $q\equiv 7 \pmod{8}$, then $m_\pi\equiv 3 \pmod{4}$. Thus $w_2(\pi)\neq 0$ for all cuspidal representations.

\medskip

Combining the above, we see that $w_2(\pi)=0$ occurs precisely in the Steinberg case, and this happens if and only if $q\equiv 7 \pmod{8}$ and $\psi\neq 1$. This completes the proof.\end{proof}

Combining the above, we have the following Proposition. 
\begin{proposition}[Stiefel--Whitney classes for self-dual $\GL_2(\F_q)$ representations]
Let $q$ be odd, let $G_q=\GL_2(\F_q)$, and let $\pi$ be a self-dual irreducible complex representation of $G_q$. Then the behavior of $w_2(\pi)$ and $w_4(\pi)$ is given by the following table.
\begin{center}
\renewcommand{\arraystretch}{1.25}
\begin{tabular}{|c|c|c|c|}
\hline
$q$ & $\pi$ & $w_2(\pi)=0$ & $w_4(\pi)=0$ \\
\hline

\multirow{4}{*}{$1 \!\!\!\pmod{4}$}
& $1$-dim 
& Yes 
& Yes \\

\cline{2-4}

& Principal series 
& $\begin{aligned}
& q\equiv 1 \!\!\!\pmod{8} \ (\chi_1(-1)=1) \\
& q\equiv 5 \!\!\!\pmod{8} \ (\chi_1(-1)=-1)
\end{aligned}$
&
$\begin{aligned}
& q\equiv 1 \!\!\!\pmod{16} \ (\chi_1(-1)=1) \\
& q\equiv 13 \!\!\!\pmod{16} \ (\chi_1(-1)=-1)
\end{aligned}$ \\

\cline{2-4}

& Steinberg 
& $q\equiv 1 \!\!\!\pmod{8}$ 
& $q\equiv 1 \!\!\!\pmod{16}$ \\

\cline{2-4}

& Cuspidal 
& $q\equiv 1 \!\!\!\pmod{8}$ 
& $q\equiv 1 \!\!\!\pmod{16}$ \\

\hline

\multirow{4}{*}{$3 \!\!\!\pmod{4}$}
& $1$-dim 
& Yes 
& -- \\

\cline{2-4}

& Principal series 
& $\begin{aligned}
& q\equiv 3 \!\!\!\pmod{8} \ (\chi_1(-1)=1) \\
& q\equiv 7 \!\!\!\pmod{8} \ (\chi_1(-1)=-1)
\end{aligned}$
& -- \\

\cline{2-4}

& Steinberg twist ($\psi\neq 1$) 
& $q\equiv 7 \!\!\!\pmod{8}$ 
& -- \\

\cline{2-4}

& Cuspidal 
& No 
& -- \\

\hline
\end{tabular}
\end{center}
\end{proposition}

\begin{lemma}
Let $G_q=\GL_2(\F_q)$ with $q$ odd, and let $\mathcal{O}_q$ denote the set of irreducible orthogonal complex representations of $G_q$. Then:

\begin{enumerate}
\item The set of one-dimensional representations has density $0$ in $\mathcal{O}_q$.
\item The set of Steinberg twists has density $0$ in $\mathcal{O}_q$.
\end{enumerate}
\end{lemma}

\begin{proof}
We estimate the sizes of the relevant families. We recall that self-dual representations of $G_q$ are orthogonal. The self-dual irreducible representations of $G_q$ consist of:
\begin{itemize}
\item $\sim \frac{q}{2}$ principal series,
\item $\sim \frac{q}{2}$ cuspidal representations,
\item $2$ one-dimensional representations,
\item $2$ Steinberg twists.
\end{itemize}
There are exactly two one-dimensional self-dual representations, namely the trivial and the quadratic character. Thus
\[
\frac{\#\{\text{1-dimensional}\}}{|\mathcal{O}_q|}
\;\ll\; \frac{1}{q}\;\longrightarrow\;0.
\]
There are exactly two Steinberg twists $\St\otimes(\psi\circ\det)$ with $\psi^2=1$. Hence
\[
\frac{\#\{\text{Steinberg twists}\}}{|\mathcal{O}_q|}
\;\ll\; \frac{1}{q}\;\longrightarrow\;0.
\]
\end{proof}

\begin{theorem}\label{main thm 2 corrected}
Let $G_q=\GL_2(\F_q)$ with $q$ odd, and let $\mathcal{O}_q$ denote the set of irreducible orthogonal complex representations of $G_q$. Then
\[
\lim_{\substack{q\to\infty\\ q\equiv a\pmod{8}}}
\frac{\#\{\pi\in \mathcal{O}_q : w_2(\pi)=0\}}{|\mathcal{O}_q|}
=
\begin{cases}
1 & \text{if } a=1,\\[4pt]
\frac14 & \text{if } a=3,5,\\[4pt]
0 & \text{if } a=7.
\end{cases}
\]

Furthermore,
\[
\lim_{X\rightarrow \infty}
\frac{\sum_{\substack{q\leq X\\
q\text{ is odd}}}\#\{\pi\in \mathcal{O}_q : w_2(\pi)=0\}}
{\sum_{\substack{q\leq X\\
q\text{ is odd}}}|\mathcal{O}_q|}
=
\frac{3}{8},
\]where in the sums above, $q$ ranges over odd prime powers.
\end{theorem}

\begin{proof}
We give explicit counts of irreducible orthogonal representations of
$G_q=\GL_2(\F_q)$ and analyze the condition $w_2(\pi)=0$ in each case.
Irreducible orthogonal representations of $G_q$ fall into three families.

\smallskip

\noindent
\emph{(i) Principal series.}
These are representations of the form
\[
\pi=\Ind_B^{G_q}(\chi\otimes \chi^{-1})
\]
with $\chi\neq \chi^{-1}$. The number of such representations is
\[
\frac{(q-1)-2}{2}=\frac{q-3}{2},
\]
since there are $q-1$ characters of $\F_q^\times$, of which exactly
two (the trivial and quadratic characters) satisfy
$\chi=\chi^{-1}$, and we identify $\chi\sim\chi^{-1}$.

\smallskip

\noindent
\emph{(ii) Cuspidal representations.}
Self-dual cuspidal representations correspond to characters
$\theta$ of $\F_{q^2}^\times$ satisfying $\theta^{q+1}=1$ and
$\theta\neq \theta^q$, modulo $\theta\sim\theta^q$.
There are $q+1$ characters with $\theta^{q+1}=1$, of which two satisfy
$\theta=\theta^q$, so the number of self-dual cuspidal representations is
\[
\frac{q-1}{2}.
\]

\smallskip

\noindent
\emph{(iii) One-dimensional and Steinberg twists.}
There are exactly two one-dimensional orthogonal representations
(the trivial and quadratic determinant characters), and exactly two
Steinberg twists
\[
\St\otimes(\psi\circ\det)
\]
with $\psi^2=1$.

\medskip

Thus
\[
|\mathcal{O}_q|
=
\frac{q-3}{2}+\frac{q-1}{2}+O(1)
=
q+O(1).
\]

We now analyze the vanishing of $w_2(\pi)$ according to the residue
class of $q \pmod{8}$.

\medskip

\noindent
\textit{(i) $q\equiv 1 \pmod{8}$.}
From Proposition~\ref{prop SW classes GL2 q=1mod4 clean},
all orthogonal representations with $\dim\pi>1$ satisfy
$w_2(\pi)=0$, and the same holds for the one-dimensional representations.
Hence
\[
\#\{\pi : w_2(\pi)=0\}
=
|\mathcal{O}_q|,
\]
so
\[
\frac{\#\{\pi : w_2(\pi)=0\}}{|\mathcal{O}_q|}
\to 1.
\]

\medskip

\noindent
\textit{(ii) $q\equiv 5 \pmod{8}$.}
In this case, $w_2(\pi)=0$ occurs precisely for principal series
representations
\[
\pi=\Ind_B(\chi\otimes\chi^{-1})
\]
with $\chi(-1)=-1$.
Among the $q-1$ characters of $\F_q^\times$, exactly
$\frac{q-1}{2}$ satisfy $\chi(-1)=-1$.
Since the trivial and quadratic characters both satisfy
$\chi(-1)=1$, all such characters are admissible.
Passing to unordered pairs $\{\chi,\chi^{-1}\}$ yields
\[
\frac{q-1}{4}
\]
principal series representations with $w_2(\pi)=0$.

All other representations (cuspidal, Steinberg, and one-dimensional)
satisfy $w_2(\pi)\neq0$. Hence
\[
\frac{\#\{\pi : w_2(\pi)=0\}}{|\mathcal{O}_q|}
=
\frac{(q-1)/4}{q+O(1)}
\to \frac14.
\]

\medskip

\noindent
\textit{(iii) $q\equiv 3 \pmod{8}$.}
In this case, $w_2(\pi)=0$ occurs precisely for the one-dimensional
representations and for principal series representations
\[
\pi=\Ind_B(\chi\otimes\chi^{-1})
\]
with $\chi(-1)=1$.
Among the $q-1$ characters of $\F_q^\times$, exactly
$\frac{q-1}{2}$ satisfy $\chi(-1)=1$.
Since $q\equiv3\pmod4$, the quadratic character satisfies
$\chi(-1)=-1$, whereas the trivial character satisfies
$\chi(-1)=1$.
Removing the trivial character leaves
\[
\frac{q-3}{2}
\]
admissible characters.
Passing to unordered pairs $\{\chi,\chi^{-1}\}$ yields
\[
\frac{q-3}{4}
\]
principal series representations with $w_2(\pi)=0$.
Together with the two one-dimensional orthogonal representations,
this gives
\[
\frac{q+5}{4}
\]
representations satisfying $w_2(\pi)=0$.

All other representations (cuspidal and Steinberg twists)
satisfy $w_2(\pi)\neq0$. Hence
\[
\frac{\#\{\pi : w_2(\pi)=0\}}{|\mathcal{O}_q|}
=
\frac{(q+5)/4}{q+O(1)}
\to \frac14.
\]

\medskip

\noindent
\textit{(iv) $q\equiv 7 \pmod{8}$.}
In this case, $w_2(\pi)=0$ occurs precisely for the nontrivial
Steinberg twist
\[
\St\otimes(\psi\circ\det),
\qquad \psi^2=1,\ \psi\neq1.
\]
Thus only $O(1)$ representations satisfy $w_2(\pi)=0$, and therefore
\[
\frac{\#\{\pi : w_2(\pi)=0\}}{|\mathcal{O}_q|}
\to 0.
\]

It remains to compute the global average over all odd prime powers.
For $a\in\{1,3,5,7\}$, define
\[
\mathcal{Q}_a(X)
=
\{q\le X : q \text{ odd prime power and } q\equiv a\pmod8\}.
\]
We first show that these four residue classes are asymptotically
equidistributed.

Restrict first to odd primes. By Dirichlet's theorem on primes in
arithmetic progressions,
\[
\pi_a(X)
:=
\#\{p\le X : p\equiv a\pmod8\}
\sim
\frac{1}{\varphi(8)}\frac{X}{\log X}
=
\frac14\frac{X}{\log X}.
\]
Hence
\[
\pi_a(X)
\sim
\pi_b(X)
\]
for all $a,b\in\{1,3,5,7\}$.

We now analyze the contribution of higher prime powers.
Let
\[
R(X)
=
\#\{p^k\le X : p \text{ odd prime},\ k\ge2\}.
\]
If $p^k\le X$ with $k\ge2$, then necessarily $p\le X^{1/2}$.
Therefore
\[
R(X)
\le
\sum_{2\le k\le \log_2 X}\pi(X^{1/k}),
\]
where $\pi(y)$ denotes the number of primes at most $y$.
Using the estimate $\pi(y)\ll y/\log y$, we obtain
\[
R(X)
\ll
\sum_{2\le k\le \log_2 X}X^{1/k}
=
O(X^{1/2}).
\]
Since the number of odd primes up to $X$ is asymptotic to $X/\log X$,
it follows that
\[
R(X)
=
o\!\left(\frac{X}{\log X}\right).
\]
Thus non-prime prime powers contribute negligibly to asymptotic density,
and consequently
\[
\#\mathcal{Q}_a(X)
\sim
\frac14
\sum_{b\in\{1,3,5,7\}}
\#\mathcal{Q}_b(X).
\]

Now set
\[
N(q)
=
\#\{\pi\in\mathcal{O}_q : w_2(\pi)=0\}
\quad\text{and}\quad
D(q)
=
|\mathcal{O}_q|.
\]
From the preceding analysis, for $q\equiv a\pmod8$ we have
\[
\frac{N(q)}{D(q)}
=
c_a+o(1),
\]
where
\[
c_1=1,
\qquad
c_3=c_5=\frac14,
\qquad
c_7=0.
\]
Moreover,
\[
D(q)=q+O(1).
\]

We decompose the denominator according to residue classes:
\[
\sum_{\substack{q\le X\\ q\text{ odd}}}D(q)
=
\sum_{a\in\{1,3,5,7\}}
\sum_{q\in\mathcal{Q}_a(X)}(q+O(1)).
\]
Since the residue classes are equidistributed and the error term is of
lower order, each congruence class contributes asymptotically one
quarter of the total. More precisely,
\[
\sum_{q\in\mathcal{Q}_a(X)} q
\sim
\frac14
\sum_{\substack{q\le X\\ q\text{ odd}}} q.
\]

Similarly,
\[
\sum_{\substack{q\le X\\ q\text{ odd}}}N(q)
=
\sum_{a\in\{1,3,5,7\}}
\sum_{q\in\mathcal{Q}_a(X)}
\bigl(c_a q+o(q)\bigr).
\]
Hence
\[
\sum_{\substack{q\le X\\ q\text{ odd}}}N(q)
\sim
\sum_{a\in\{1,3,5,7\}}
c_a
\sum_{q\in\mathcal{Q}_a(X)} q.
\]
Using the equidistribution just proved, we obtain
\[
\sum_{\substack{q\le X\\ q\text{ odd}}}N(q)
\sim
\frac14(c_1+c_3+c_5+c_7)
\sum_{\substack{q\le X\\ q\text{ odd}}} q.
\]

Since
\[
\sum_{\substack{q\le X\\ q\text{ odd}}}D(q)
\sim
\sum_{\substack{q\le X\\ q\text{ odd}}} q,
\]
it follows that
\[
\lim_{X\to\infty}
\frac{\sum_{\substack{q\le X\\ q\text{ odd}}}N(q)}
{\sum_{\substack{q\le X\\ q\text{ odd}}}D(q)}
=
\frac14(c_1+c_3+c_5+c_7).
\]
Substituting the values of the constants gives
\[
\frac14\left(1+\frac14+\frac14+0\right)
=
\frac38.
\]
This proves the theorem.
\end{proof}
\bibliographystyle{alpha}
\bibliography{references}

\end{document}